\documentclass[10pt]{amsart}
\usepackage[utf8]{inputenc}
\usepackage{bbm}
\usepackage{graphicx}
\usepackage{geometry}
\usepackage[
backend=biber,
style=alphabetic,
]{biblatex}
\usepackage{listings}
\usepackage{bigints}
\usepackage{color}
\usepackage{amssymb}
\usepackage{amsmath}
\usepackage{mathtools}
\usepackage{tcolorbox}
\usepackage{hyperref}
\usepackage{enumitem}  
\usepackage{mathabx}
\usepackage{amsthm} 
\usepackage{bm}
\usepackage {units}

\title{Dirichlet form analysis of the Jacobi process}
\date{\today}
\author{Martin Grothaus \textsuperscript{1}}
	\thanks{\textsuperscript{1} Department of Mathematics, TU Kaiserslautern, PO Box 3049, 67653 Kaiserslautern}
\author{Max Sauerbrey  \textsuperscript{2}} 
\thanks{\textsuperscript{2} Delft Institute of Applied Mathematics, TU Delft, Mekelweg 4,
	2628 CD  Delft. \\\textit{Email addresses}: \href{mailto:Grothaus@mathematik.uni-kl.de}{Grothaus@mathematik.uni-kl.de} , \href{mailto:M.Sauerbrey@tudelft.nl}{M.Sauerbrey@tudelft.nl}}

\theoremstyle{plain}
\newtheorem{thm}{Theorem}[section]
\newtheorem{cor}[thm]{Corollary}
\newtheorem{lemma}[thm]{Lemma}
\newtheorem{prop}[thm]{Proposition}

\theoremstyle{definition}
\newtheorem{defi}[thm]{Definition}

\newtheorem{rem}[thm]{Remark}

\DeclareMathOperator{\loc}{loc}
\DeclareMathOperator{\codim}{codim}

\DeclareMathOperator{\Ca}{Cap}
\DeclareMathOperator{\spa}{span}

\DeclareMathOperator{\EE}{\mathcal{E}}
\DeclareMathOperator{\FF}{\mathcal{F}}
\DeclareMathOperator{\RR}{\mathbb{R}}
\DeclareMathOperator{\sgn}{sign}

\numberwithin{equation}{section}
\begin{document}
	\maketitle
	
	\vspace{3cm}
	\begin{center}
		\textbf{
		Abstract
}	\end{center}
	We construct and analyze the Jacobi process - in mathematical biology referred to as Wright-Fisher diffusion - using a Dirichlet form. 
	The corresponding Dirichlet space takes the form of a Sobolev space with different weights for the function itself and its derivative.
	Depending on the parameters we characterize the boundary behavior of the functions in the Dirichlet space, show density results, derive Sobolev embeddings and verify functional inequalities of Hardy type. Since the generator is a hypergeometric differential operator, many of the proofs can be carried out by explicit calculations involving hypergeometric functions. We deduce corresponding properties for the associated semigroup and Markov process and show that the latter is up to minor technical modifications a solution to the Jacobi SDE.
	
	\vspace{1cm}

	\tableofcontents
	
	\newpage
	
	\section{Introduction}
	The Jacobi process is a $[0,d]$-valued solution to the stochastic differential equation 
	\begin{equation}\label{eq3n}
		dY_t\,=\,(a-bY_t) \,dt\,+\,\sigma \sqrt{Y_t(d-Y_t)}\,dW_t.
	\end{equation}
	Here, $a,b$ and $\sigma, d>0$ are parameters and $W$ a Brownian motion. The Jacobi process arises in different applications, most prominently as a model for allele frequencies in mathematical biology, see \cite[Section 10.2, pp.415-426]{ethier2005markov}, where it is commonly referred to as Wright-Fisher diffusion. Moreover, the Jacobi processes can be used as a model for membrane depolarization \cite{neural_model} and interest rates \cite{interest_rates} or in the modeling of electricity prices \cite{electricity}. 
	
	The parameters 	\begin{equation}\label{EqAlphaBeta}
		\alpha=\frac{2b}{\sigma^2}-\frac{2a}{\sigma^2d}-1\;\;\text{and}\;\;\beta=
		\frac{2a}{\sigma^2d}-1
	\end{equation}
	capture the behavior of the process close to the boundary points. Indeed, if $Y_t$ is close to $0$, the drift is  approximately given by $a dt$ and the stochastic fluctuation by  $\sigma (Y_td )^\frac{1}{2} dW_t$ which has variance $\sigma^2 (Y_td)  dt$. Therefore, the ratio $\frac{2a}{\sigma^2d}$ quantifies how much the drift pushes the process  back into the state space $[0,d]$ compared to the stochastic fluctuations. This demonstrates the importance of the parameter $\beta$ and an analogous argument shows that $\alpha$ quantifies the boundary behavior near d
	
	If $\alpha, \beta >-1$, the Jacobi process is one of  three types of real-valued diffusions, which are associated to a family of orthogonal polynomials, see \cite{mazet}. This allows for an explicit expression of the transition semigroup, which can be used as in \cite
	{interest_rates} to analyze the process. The Jacobi process can be constructed by its associated Feller semigroup \cite{ethier2005markov} and belongs to the larger class of  Pearson diffusions \cite{sorensen2008}. For an extensive study of the associated differential operator we refer to \cite{Epstein_Mazzeo2009}.
Examples of how to analyze the Jacobi process using classical methods for one-dimensional diffusions can be found in  \cite{Huillet_2007}. There are many other works on the Jacobi process, but we hope that this  selection gives an overview of the main tools, which were used to construct and analyze the Jacobi process up to now. In the current article, we take the new approach to construct and analyze solutions to \eqref{eq3n} by Dirichlet form methods. We stress that we  allow for the case $\alpha \le -1\vee \beta \le -1$ which leads to additional mathematical challenges.

As state space of the Dirichlet form we define $X$ as the union of the interval $(0,d)$ with the right boundary point $\{d\}$ if $\alpha>-1$ and with the left boundary $\{0\}$ if $\beta>-1$. 
We write $\mathfrak{B}$ for the Borel $\sigma$-field on $X$ and $dx$ for the Lebesgue measure. 
The  generator of \eqref{eq3n} is given by
\begin{equation}\label{eq12}
	Gf(x)\,=\,\frac{1}{2}\sigma^2x(d-x)f''(x)+(a-bx)f'(x)
\end{equation} for $f\in C^2([0,d])$. The stationary solution to the corresponding Kolmogorov forward equation on $(0,d)$ is
\[
m(x)=\frac{x^{\beta}(d-x)^{\alpha}}{d^{\alpha+\beta+1}}\] and is the natural candidate for the invariant measure of the Jacobi process. Therefore, we equip $X$ with the measure with density $m$
with respect to $dx$. Then,
$dm=mdx$ is  a positive Radon measure on $X$ with full support and we define $dcm=cmdx$ using the additional density  $
c(x)\,=\,\frac{1}{2}\sigma^2x(d-x)$. The calculation
\begin{align}\begin{split}\label{eq37}
		(cm)'(x)\,=\,
		\frac{1}{2}\sigma^2\left[\frac{2a}{\sigma^2d}(d-x)-{\left(\frac{2b}{\sigma^2}-\frac{2a}{\sigma^2d}\right)}x\right]\frac{x^\beta(d-x)^{\alpha}}{d^{\alpha+ \beta+1}}\,=\,(a-bx)m(x)
	\end{split}
\end{align}
implies that
\[
(f'cm)'(x)\,=\,\big[c(x)f''(x)+(a-bx)f'(x)\big]m(x)\,=\, m(x)Gf(x)
\]
for  $f\in C_c^\infty((0,d))$. Consequently,
the operator  $(G, C_c^\infty((0,d)))$ is of the form \cite[Equation (3.3.17), p.134]{fukushima2011}. By \cite[Theorem 3.3.1, p.135]{fukushima2011}
\begin{gather}\begin{split}\label{eq72}
		D(\mathcal{E})\,=\,\big\{
		f\in L^2(X,dm)\big|\, f'\in L^2( X,dcm)
		\big\},\\
		\mathcal{E}\colon 	D(\mathcal{E})\times 	D(\mathcal{E})\to \mathbb{R}, (f,g)\mapsto \,\int_X f'(x)g'(x)\, dcm(x)
	\end{split}
\end{gather}
defines a strongly local Dirichlet form corresponding to the maximal markovian self-adjoint extension $(L,D(L))$ of  $(G, C_c^\infty((0,d)))$ in $L^2(X,dm)$. 
We denote the markovian, symmetric, strongly continuous semigroup of contractions generated by $(L,D(L))$ by $(T_t)_{t>0}$ and write
\[
\EE_\lambda (f,g)\,=\, \mathcal{E}(f,g)+\lambda (f,g)_{L^2(X,dm)}.
\]
for $\lambda>0$ and $f,g\in D(\EE)$. We equip $D(\EE)$ with the topology induced by any of the inner products $\EE_\lambda$ and, if considered as a Hilbert space, we  equip it with $\mathcal{E}_1$ unless stated otherwise. 
We write  $\FF$ for the closure of $C_c^\infty((0,d))$ in $D(\EE) $.

In the preliminary Section \ref{Sec_local_sol} we introduce and analyze a notion of  solutions to \eqref{eq3n}, which will serve later on to translate our findings on Hunt processes to statements on  general solutions to \eqref{eq3n}. 
In Section \ref{Sec_Dirichlet_space} we consider  the Dirichlet form $(\EE,D(\EE))$, prove basic properties and characterize under which assumptions on the parameters we have $D(\EE)=\FF$. We also answer the related question of how functions from $D(\EE)$ behave near the boundary points and  prove embedding theorems and  functional inequalities in the space $D(\EE)$. 
In Section \ref{Sec_Semigroup} we analyze  $(T_t)_{t>0}$ in the context of Markovian semigroups.
Moreover, we quantify the spectral gap for  special cases of $\alpha, \beta$.
Finally, in Section \ref{Sec_process} we translate the findings on  $(\EE,D(\EE))$ and $(T_t)_{t>0}$ into properties of an associated Hunt process. We show under which assumption on the parameters, the process is recurrent, ergodic, conservative or a variant of transitive.  
Furthermore, we show how a Markov process associated to $(\EE, D(\EE))$ as well as its restriction to $(0,d)$ is related to general solutions to \eqref{eq3n}. This allows to transfer the gathered results from this article as well as future findings on  $(\EE, D(\EE))$ into statements on \eqref{eq3n}.

 \section{Local solutions to the Jacobi SDE
	 }\label{Sec_local_sol}
	Throughout this article we fix parameters $a,b\in\RR $ and $\sigma,d>0$. We note, that equality of random variables is meant almost surely unless stated otherwise.
	\begin{defi}\label{local_sol}
		A  local solution to the stochastic differential equation
		\eqref{eq3n} is  a quadruple, consisting of a filtered probability space $(\Omega, \mathfrak{A}, P, \mathfrak{F})$,  a Brownian motion $W$, a real-valued, adapted process $Y$ and a stopping time  $\zeta$, such that  $\mathfrak{F}$ satisfies the usual conditions
		and the following conditions are satisfied.
		\begin{enumerate}[label=(\roman*)]\item The mapping $Y_\cdot(\omega)\colon[0,\zeta(\omega)]\to \mathbb{R}$ is $[0,d]$-valued and continuous for every $\omega \in \Omega$.
			\item For all $t\ge 0$ we have that \begin{equation}\label{eq6n}
				Y_{t\wedge \zeta}\, =\, Y_0 \,+\, \int_0^{t\wedge \zeta} (a-bY_s) \,ds \, +\, 
				\int_0^{t\wedge \zeta} \sigma\sqrt{Y_s(d-Y_s)} \,dW_s.
			\end{equation} 
			\item It holds $P(\{Y_\zeta \notin \{0,d\}\}\cap \{\zeta<\infty\})=0$.
		\end{enumerate}
	\end{defi}
	We sometimes call $(Y, \zeta)$ a local solution and $(\Omega, \mathfrak{A}, P, \mathfrak{F},W)$ its stochastic basis or do not even specify the latter. 
	\begin{rem}\label{wdef_rem}We understand 
		the first integral in \eqref{eq6n} as the continuous and adapted process
		\[
		\int_0^\cdot \mathbbm{1}_{[0, \zeta]}(s)(a-bY_s) \,ds.
		\]at time $t$.
		The second integral is meant as the stochastic integral
		\[
		\int_0^t \mathbbm{1}_{[0,\zeta]}(s)\sigma\sqrt{Y_s(d-Y_s)}\,dW_s,
		\]
		which is well-defined, since the integrand is left-continuous, adapted and bounded.
	\end{rem}
	
	In the next definition $(Y, \zeta)$ and  $(\tilde{Y}, \tilde{\zeta})$ are local solutions on the same stochastic basis.
	\begin{defi}
	The solution $(\tilde{Y}, \tilde{\zeta})$ is an extension of  $(Y, \zeta)$, if $\tilde{\zeta}\ge \zeta$ and $Y_{t\wedge \zeta}=\tilde{Y}_{t\wedge \zeta}$ for all $t\ge 0$.  If additionally $\zeta=\tilde{\zeta}$, we identify the local solutions.
\end{defi}
If $(\tilde{Y}, \tilde{\zeta})$ is an extension of  $(Y, \zeta)$, we write $(Y, \zeta)\lesssim (\tilde{Y}, \tilde{\zeta})$. In the case of equality we write $(Y, \zeta)= (\tilde{Y}, \tilde{\zeta})$. The relation $\lesssim$ is then a partial ordering on the  equivalence classes of local solutions on a fixed stochastic basis.
	\begin{defi}
		A local solution $(Y,\zeta)$ is called minimal (maximal), if it is a minimal (maximal) element with respect to the ordering $\lesssim$ of local solutions.
	\end{defi}
	To use the theory of stochastic differential equations on $\RR$ we
	let $\mu\in C_c^\infty(\RR)$ such that $\mu(x)= (a-bx)$ on a neighborhood of $[0,d]$ and $\nu(x)= \mathbbm{1}_{[0,d]}(x)\sigma\sqrt{x(d-x)}$ for $x\in \mathbb{R}$. Then $\mu$ is Lipschitz continuous and $\nu$ is $\nicefrac{1}{2}$-H\"older continuous. Therefore, the stochastic differential equation
	\begin{equation}\label{eq4n}
		dZ_t\, = \,\mu(Z_t) dt\,+\,\nu (Z_t) dW_t
	\end{equation}
	is well-posed. Indeed, existence of weak solutions follows by the Skohorod existence theorem \cite[Theorem 18.7, p.341; Theorem 18.9, p.342]{kallenberg1997foundations} and pathwise uniqueness holds due to the Yamada-Watanabe condition \cite[Theorem 20.3, p.374]{kallenberg1997foundations}. Consequently, the Yamada-Watanabe theorem \cite[Lemma 18.17]{kallenberg1997foundations} implies that  strong existence and uniqueness in law holds for arbitrary initial distributions on $\RR$. 
	\begin{thm}\label{extension_thm}
		Let $(Y,\zeta)$ be a local solution with respect to a stochastic basis
		$(\Omega, \mathfrak{A}, P, \mathfrak{F},W)$. Then the unique solution $Z$ to \eqref{eq4n} with inital value $Y_0$ satisfies
		\begin{equation}\label{eq5n}
			Z_{t\wedge \zeta}=Y_{t\wedge \zeta}
		\end{equation}
		for all $t\ge 0$. Moreover, 
		\begin{enumerate}[label=(\roman*)]
			\item $(Y,\zeta)$ is minimal iff $\zeta= \inf\{t\ge 0| Z_t\in \{0,d\}\}$,   
			\item  $(Y,\zeta)$ is maximal iff $\zeta= \inf\left\{t\ge 0| Z_t\notin  [0,d] \right\}$.
		\end{enumerate}
	\end{thm}
	\begin{proof}
		Let $(Y,\zeta)$ be a local solution, then the first part of the claim follows by a localized Yamada-Watanabe condition, see Theorem \ref{yamada_watanabe}. We  define $\tilde{\zeta}$ as the infinum in (i),  which defines a stopping time. In particular $(Z,\tilde{\zeta})$ is a local solution to \eqref{eq3n}.
		 Due to condition (iii) of Definition \ref{local_sol} we have that
		\begin{align*}
			P(\{\tilde{\zeta}>\zeta\})=P(\{\tilde{\zeta}>\zeta\}\cap \{Y_\zeta\in \{0,d\}\}).
		\end{align*}
		By \eqref{eq5n} it follows that $P(\{Z_\zeta \ne Y_\zeta\}\cap\{\zeta<\infty\})=0$ and therefore $P(\{\tilde{\zeta}>\zeta\})$  is dominated by
		\[
		P(\{\tilde{\zeta}>\zeta\}\cap \{Z_\zeta\in \{0,d\}\})=0.
		\]
	Hence $(Z,\tilde{\zeta})\lesssim (Y,\zeta)$.
		Therefore, if $\zeta \ne \tilde{\zeta}$, the local solution $(Y,\zeta)$ is not minimal. Conversely, if  $\zeta=\tilde{\zeta}$, it holds $(Z,\tilde{\zeta}) = (Y, \zeta) $ and by our previous considerations it follows $(Z, \tilde{\zeta})\lesssim (\hat{Y}, \hat{\zeta})$ for every other local solution $(\hat{Y}, \hat{\zeta})$.
		
		Next, let $\tilde{\zeta}$ be the infimum from (ii) instead. 
		The identity \eqref{eq5n} implies that
		\[
		P(\{\tilde{\zeta}<\zeta\})\le 
		P(\{\inf \{t|Z_{t\wedge \zeta}\notin [0,d]\}<\infty\})\le P(\{\exists t\ge 0: Y_{t\wedge \zeta}\notin [0,d]\})=0.
		\]
		In the latter equality we used condition (i) of Definition \ref{local_sol}.
		We conclude $(Y,\zeta)\lesssim (Z, \tilde{\zeta})$  and obtain (ii)  analogously to (i).
	\end{proof}
	The previous statement implies pathwise existence and uniqueness of minimal and maximal local solutions to \eqref{eq3n} with a prescribed initial value. To formulate a uniqueness statement concerning their laws, we introduce the space $[0,d]_{\Delta}$ as the set $[0,d]\cup\{\Delta\}$, where $\Delta$ is topologically adjoined  as a separate point. We equip the space $\mathcal{X}=([0,d]_{\Delta})^{[0,\infty)}$  with the corresponding product of Borel $\sigma$-fields. 
	\begin{cor}\label{uniq_laws}
		Let $((\Omega^{(i)}, \mathfrak{A}^{(i)}, P^{(i)}), \mathfrak{F}^{(i)},W^{(i)}, Y^{(i)}, \zeta^{(i)} )$ for $i\in \{1,2\}$ be two minimal (maximal) local solutions to \eqref{eq3} with the same inital distribution  on $[0,d]$. Then the
		laws $P^{(i)}\circ (\tilde{Y}^{(i)})^{-1}$ on $\mathcal{X}$ coincide, where
		\[
		\tilde{Y}^{(i)}_t(\omega)=\begin{cases}
			Y_t(\omega), &t\le \zeta^{(i)}(\omega)
			\\\Delta, &t> \zeta^{(i)}(\omega),
		\end{cases} \;\; \;\;\omega \in \Omega^{(i)}.
		\] In particular the laws of their lifetimes $P^{(i)}\circ (\zeta^{(i)})^{-1} $ coincide.
	\end{cor}
	\begin{proof}
		We first consider the case that the two solutions are maximal. Let $Z^{(i)}$ be the solution to \eqref{eq4n} with the same initial value and
		$\tilde{\zeta}^{(i)}$ the stopping time as in Theorem \ref{extension_thm} (ii)
		for $i\in\{1,2\}$. By the uniqueness in law of \eqref{eq4n}, we have that $P^{(1)}\circ (Z^{(1)})^{-1}=P^{(2)}\circ (Z^{(2)})^{-1}$ as probability measures on $C([0,\infty),\RR)$, where we equip the latter space with its Borel $\sigma$-field.
		We note that
		\[\tau\colon C([0,\infty),\RR)\to \mathbb{R}\cup\{\infty\},
		f\mapsto \inf\left\{t\ge 0| f(t)\notin [0,d]\right\}
		\]
		is measurable.
		 Therefore, the mapping $\phi \colon C([0,\infty),\RR)\to \mathcal{X}$, where
		\begin{align*}
			\phi(f)(t)= \begin{cases}
				f(t),& t\le \tau,\\
				\Delta,& t> \tau\\
			\end{cases}
		\end{align*}
		is measurable as well.
		Consequently, if we denote the measure $P^{(i)}\circ (Z^{(i)})^{-1}$ by $Q$, which is independent of i, we obtain that
		\begin{equation}\label{eq11}
			P^{(i)}\circ (\tilde{Y}^{(i)})^{-1}=  	Q\circ\phi^{-1},
		\end{equation}by Theorem \ref{extension_thm}. The statement for minimal solutions follows analogously by replacing $\tau $ by
		\[
		f\,\mapsto\, \inf\left\{t\ge 0| f(t)\in \{0,d\}\right\}.
		\]
	\end{proof}

	\section{The Dirichlet space}\label{Sec_Dirichlet_space}
	We recall that $(\EE, D(\EE))$ was defined by \eqref{eq72}. Since the density functions $m$ and $cm$ are both bounded from below on every compactly contained subset of $(0,d)$, $D(\EE)$ embeds continuously in the local Sobolev space $ H^1_{\loc}((0,d))$, for an introduction to these spaces we refer to \cite{burazin2005}. In particular, by the Sobolev embedding theorem every $f\in D(\EE)$ admits a continuous version on $(0,d)$, which we denote by $f^*$. Moreover, we obtain  the following integration by parts formula
		\begin{equation}\label{int_by_parts}
	f^*(r)g(r)-f^*(l)g(l)\,=\,\int_{l}^{r}f'g(x)+fg'(x) \,dx
	\end{equation} for any function $g$ which is continuously differentiable on $[l,r]\subset (0,d)$ as a consequence of \cite[Corollary 8.10, p.215]{brezis2010functional}.
	\subsection{Orthogonal decomposition of the domain}
	
	We recall the possibly smaller closed subdomain $\FF\subset D(\EE)$. 
For each $\lambda>0$ there is an orthogonal decomposition
	$
	D(\EE)\,=\, \FF \perp \FF^\perp_\lambda$
	with respect to the inner product $\EE_\lambda$. This section is devoted to characterizing $\FF$ and $\FF_\lambda^\perp$. We use a duality argument and properties of hypergeometric functions and start with a technical observation.
		\begin{lemma}\label{embedding}
		Let $f\in \mathcal{F}$. Then we have  that
		\begin{equation}\label{eq8}
			\begin{cases*}
				\lim_{x\nearrow d} f^*cm(x)\,=\,0,&$\alpha>-1$,\\
				\lim_{x\nearrow d} f^*(x)\,=\,0,&$\alpha\le-1$
			\end{cases*}
		\end{equation}
		and
		\begin{equation*}
			\begin{cases*}
				\lim_{x\searrow 0} f^*cm(x)\,=\,0,&$\beta>-1$,\\
				\lim_{x\searrow 0} f^*(x)\,=\,0,&$\beta\le-1$.
			\end{cases*}
		\end{equation*}
	\end{lemma}
	\begin{proof}
	If $\alpha\le -1$ the density functions $m$ and $cm$ are   bounded away from zero on $\left(\frac{d}{2}, d\right)$. Hence the mapping
		\[
		D(\mathcal{E})\to H^1\left(\left(\frac{d}{2}, d\right)\right),\,f\mapsto f|_{\left(\frac{d}{2}, d\right)}
		\]
		is continuous, where we write $H^1$ for the classical first-order Sobolev space. Then \eqref{eq8} follows by approximation and the Sobolev embedding theorem since $C_c^\infty(0,d)$ is by definition dense in $\FF$.
		Secondly, we assume $\alpha >-1$. Again by an approximation argument, \eqref{eq8} follows if we can verify continuity of 
		\begin{equation}\label{eq32}
			\left(C_c^\infty((0,d)), \mathcal{E}_1\right)\to C^0\left(\left[\frac{d}{2},d\right]\right),\,\varphi\mapsto cm\varphi|_{\left[\frac{d}{2},d\right]}.
		\end{equation}By \eqref{eq37} we have that
		\begin{equation}\label{eq3}
			(\varphi cm  )'(x)\,=\,cm(x)\varphi'(x)+(a-bx)m(x)\varphi(x)
		\end{equation}
	for every $\varphi\in C_c^\infty((0,d))$ and therefore
		\begin{align*}&
			|cm \varphi(x)|\,\le \,\int_{\frac{d}{2}}^d |\varphi'|\,dcm\,+\,\big(|a|+|b|d\big)\int_{\frac{d}{2}}^d |\varphi|\,dm\\\le \,&\left[\left(
			\int_{\frac{d}{2}}^d \mathbbm{1}\,dcm\right)^\frac{1}{2}\,+\,
			(|a|+|b|d)\left({\int_{\frac{d}{2}}^d \mathbbm{1}\,dm}\right)^\frac{1}{2}\right]\,\sqrt{\mathcal{E}_1(\varphi,\varphi)}
		\end{align*}
		for any $x\in \left[\frac{d}{2},d\right]$. The prefactor on the right-hand side is finite since $\alpha>-1$ and consequently \eqref{eq32} is indeed bounded.
		The second part of the statement can be proved analogously.
	\end{proof}

	 We proceed by noting that $f\in \FF_\lambda^\perp$ if and only if
	\[
	\forall \varphi\in C_c^\infty((0,d)):\;\;\int_0^d f'(x)\varphi'(x)\,dcm(x) +		\lambda \int_0^d f(x)\varphi(x)\, dm(x) \,=\,0,
	\]which  is a weak formulation of the differential equation $Gf=\lambda f$ due to \eqref{eq3}, where we recall that $G$ was defined by \eqref{eq12}. Hence, the solution space to this differential equation is a proper candidate for $\FF^\perp_\lambda$. To simplify the notation we introduce the $\lambda$-dependent parameter
	\[
	\gamma\,=\,\frac{\sqrt{\sigma^4-4b\sigma^2-8\lambda \sigma^2+4b^2}}{2\sigma^2}
	\] such that by definition
	\begin{equation}\label{eq21}
		\left(\frac{\alpha+\beta+1}{2}\right)^2-\gamma^2\,=\,\frac{2\lambda }{\sigma^2}.
	\end{equation}
	Using $_2F_1$ as notation for hypergeomtetric functions, we define the real-valued function
	\begin{equation*}\label{eq25}
		\xi_\lambda(x)\,=\,\begin{cases*}
			{}_2F_1\left(
			\frac{\alpha+\beta+1}{2}+\gamma,
			\frac{\alpha+\beta+1}{2}-\gamma;\beta+1;\frac{x}{d}
			\right),&$\beta >-1$,\\\left(\frac{x}{d}\right)^{-\beta}
			{}_2F_1\left(
			\frac{\alpha-\beta+1}{2}+\gamma,
			\frac{\alpha-\beta+1}{2}-\gamma;1-\beta;\frac{x}{d}
			\right),&$\beta \le-1$
		\end{cases*}
	\end{equation*}
on $(0,d)$. Since $\xi_\lambda$ is a rescaled version of the solution to a hypergeometric differential equation, see \cite[p.163]{lebedev2012special}, it satisfies indeed  $G\xi_\lambda= \lambda \xi_\lambda$ on $(0,d)$. We provide an observation on the monotonicity of $\xi_\lambda$.
	
	\begin{prop}\label{mon} If $\alpha\le -1$ or $\beta\le -1$ we assume additionally that
		\begin{equation}\label{ass_r}\lambda \,> \,\frac{\sigma^2}{2}
			\left(
			\frac{\alpha+\beta+1}{2}
			\right)^2.\end{equation}Then $\xi_\lambda$ is strictly monotonically increasing.
	\end{prop}
	\begin{proof}
		We consider first the case $\beta>-1$ in which we  write
		\begin{equation}\label{eq174}
			\xi_\lambda(x)=\sum_{k=0}^\infty\left(\frac{x}{d}\right)^k\, \prod_{j=1}^{k}\frac{\left(\frac{\alpha+\beta+1}{2}+\gamma+j-1\right)\left(\frac{\alpha+\beta+1}{2}-\gamma+j-1\right)}{j(\beta+j)} .
		\end{equation}
		If $\alpha>-1$ we have for all $j\in\mathbb{N}$  that
		\[
		\left(\frac{\alpha+\beta+1}{2}+j-1\right)^2-\gamma^2\,\ge\,
		\left(\frac{\alpha+\beta+1}{2}\right)^2-\gamma^2\,=\,\frac{
			2\lambda}{\sigma^2}\,>\,0.
		\]
		If $\alpha\le -1$ it follows by \eqref{eq21} and \eqref{ass_r} that $-\gamma^2>0$ and hence 
		\[
		\left(\frac{\alpha+\beta+1}{2}+j-1\right)^2-\gamma^2\,>\,0\]
		for each $j\in\mathbb{N}$. In both cases the coefficients of \eqref{eq174} are  strictly positive and the claim follows.
		For $\beta \le -1$ it suffices by analogous arguments to realize that  \eqref{ass_r} implies that
		\[
		\left(\frac{\alpha-\beta+1}{2}+j-1\right)^2-\gamma^2\,>\,0
		\]
		for any $j\in\mathbb{N}$.
	\end{proof}	 
	It will be also handy to calculate some boundary values of $\xi_\lambda$ and its derivative. 
	The following formulas are the key tool to perform the explicit calculations and can be found in \cite[Theorem 2.1.3, p.63; Theorem 2.2.2, p.66]{andrews1999special}. The appearing function $\Gamma$ is the well-known $\Gamma$-function.
	
	\begin{lemma}\label{gauss_limit} Let $\kappa, \iota, \upsilon \in \mathbb{C}$ with $-\upsilon\notin \mathbb{N}_0$. 
		\begin{enumerate}[label=(\roman*)]
			\item If $\mathfrak{R}(\upsilon-\kappa-\iota)>0$, then
			\begin{equation}\label{eq17}\lim_{x\nearrow 1}
				{}_2F_1(\kappa, \iota;\upsilon;x )=
				\frac{\Gamma(\upsilon)\Gamma(\upsilon-\kappa-\iota)}{\Gamma(\upsilon-\kappa)\Gamma(\upsilon-\iota)}.
			\end{equation}
			\item 	If $\upsilon-\kappa-\iota=0$, then
			\begin{equation}\label{eq18}
				\lim_{x\nearrow 1}\frac{{}_2F_1(\kappa, \iota;\upsilon;x )}{-\log(1-x)}=
				\frac{\Gamma(\upsilon)}{\Gamma(\kappa)\Gamma(\iota)}.
			\end{equation}
			\item If $\mathfrak{R}(\upsilon-\kappa-\iota)<0$, then
			\begin{equation}\label{eq26}
				\lim_{x\nearrow 1}\frac{{}_2F_1(\kappa, \iota;\upsilon;x )}{(1-x)^{\upsilon-\kappa-\iota}}=
				\frac{\Gamma(\upsilon)\Gamma(\kappa+\iota-\upsilon)}{\Gamma(\kappa)\Gamma(\iota)}.
			\end{equation}
		\end{enumerate}
	\end{lemma}
They result in the following statements on $\xi_\lambda$. The detailed calculations leading to it are contained in Appendix \ref{app_BV}
	\begin{lemma}\label{prop_f1}The function $\xi_\lambda$ admits the following properties. \begin{enumerate}[label=(\roman*)]
			\item \label{l_boundary}
		It holds that
		\begin{equation*}\label{eq5}
			\lim_{x\searrow 0} \xi_\lambda(x)=
			\begin{cases*}1,&$\beta>-1$,
				\\
				0,&$\beta \le -1$.
			\end{cases*}
		\end{equation*}
		\item \label{r_boundary} It holds that
		\begin{equation*}\label{eq4}
			\lim_{x\nearrow d} \xi_\lambda(x)\,=\,
			\begin{cases*}
				\frac{\Gamma(\beta+1)\Gamma(-\alpha)}{\Gamma\left(\frac{-\alpha+\beta+1}{2}+\gamma\right)\Gamma\left(\frac{-\alpha+\beta+1}{2}-\gamma\right)}
				,&$\alpha<0,\;\beta >-1$
				\\
				\frac{\Gamma(1-\beta)\Gamma(-\alpha)}{\Gamma\left(\frac{-\alpha-\beta+1}{2}+\gamma\right)\Gamma\left(\frac{-\alpha-\beta+1}{2}-\gamma\right)}
				,&$\alpha<0,\;\beta \le -1$
				\\
				\infty,&$\alpha \ge 0$
			\end{cases*}
		\end{equation*}
	and the above limit is positive if it is finite.
	\item \label{l_boundary2}It holds that
	\begin{equation*}\label{eq28}
		\lim_{x\searrow 0}\, \xi_\lambda'cm(x)\,=\,
		\begin{cases*}
			0,&$\beta>-1$,
			\\
			\frac{-\beta \sigma^2}{2},&$\beta \le -1$.
		\end{cases*}
	\end{equation*}
\item\label{r_boundary2} Under the additional assumption
	\begin{equation}\label{ass_r2}\lambda \,> \,\frac{\sigma^2}{2}
	\left(
	\frac{\alpha+\beta+1}{2}
	\right)^2.\end{equation}
for $\alpha< -1$ and $\beta\le -1$ it holds that
\begin{equation*}\label{eq29}
	\lim_{x\nearrow d}\, \xi_\lambda'cm(x)\,=\,
	\begin{cases*}
		\frac{\lambda\Gamma(\beta+1)\Gamma(\alpha+1)}{\Gamma\left(
			\frac{\alpha+\beta+3}{2}+\gamma
			\right)
			\Gamma\left(
			\frac{\alpha+\beta+3}{2}-\gamma
			\right)},&$\alpha >-1,\;\beta>-1$
		,\\\left[\frac{2\lambda }{\sigma^2}-\beta(\alpha+1)\right]\frac{\sigma^2\Gamma(1-\beta)\Gamma(\alpha+1)}{2\Gamma\left(\frac{\alpha-\beta+3}{2}+\gamma\right)\Gamma\left(\frac{\alpha-\beta+3}{2}-\gamma\right)},&$\alpha >-1,\;\beta\le-1$,\\
		\infty,&$\alpha \le-1$
	\end{cases*}
\end{equation*}
and the above limit is positive if it is finite.
	\end{enumerate}
	\end{lemma}

	To obtain a second solution to $Gf=\lambda f$ we define the dual set of parameters	\[
	(d^\dag,a^\dag,b^\dag,\sigma^\dag)=(d,bd-a,b,\sigma).
	\]
	Then, accordingly
	\[
	G^\dag f(x)\,=\,\frac{1}{2}\sigma^2x(d-x)f''(x)+((bd-a)-ax)f'(x)
	\]
	such that 
	$Gf=\lambda f$ is equivalent to $G^\dag (f(d-\cdot ))=\lambda f(d-\cdot)$. Hence, the function $\eta_\lambda=\xi_\lambda^\dag(d-\cdot)$ satisfies $G\eta_\lambda=\lambda\eta_\lambda$. It is straightforward to verify that $\alpha^\dag=\beta
	$, $\beta^\dag=\alpha$ and $\gamma^\dag=\gamma$ such that we can write explicitly
	\begin{equation*}
		\eta_\lambda(x)=\begin{cases*}
			{}_2F_1\left(
			\frac{\alpha+\beta+1}{2}+\gamma,
			\frac{\alpha+\beta+1}{2}-\gamma;\alpha+1;1-\frac{x}{d}
			\right),&$\alpha >-1$\\\left(1-\frac{x}{d}\right)^{-\alpha}
			{}_2F_1\left(
			\frac{-\alpha+\beta+1}{2}+\gamma,
			\frac{-\alpha+\beta+1}{2}-\gamma;1-\alpha;1-\frac{x}{d}
			\right)&$\alpha \le-1$
		\end{cases*}
	\end{equation*}for $x\in (0,d)$. In the following corollary we collect all the properties of $\eta_\lambda$, which follow immediately from the respective properties of $\xi_\lambda^\dag$.
	
	\begin{cor}\label{f2}The function $\eta_\lambda$  admits the following properties.
		\begin{enumerate}[label=(\roman*)]
			\item We have 
			$G\eta_\lambda=\lambda\eta_\lambda
			$
			on $(0,d)$.
			\item Under the additional assumption
			\[\lambda \,> \,\frac{\sigma^2}{2}
			\left(
			\frac{\alpha+\beta+1}{2}
			\right)^2\]
			for $\alpha\le -1$ or $\beta\le -1$ the function $\eta_\lambda$ is strictly monotonically decreasing.
			
			\item It holds that
			\begin{equation*}\label{eq6}
				\lim_{x\nearrow d} \eta_\lambda(x)\,=\,
				\begin{cases*}1,&$\alpha>-1$,
					\\
					0,&$\alpha \le -1$.
				\end{cases*}
			\end{equation*}
			\item It  holds that
			\begin{equation*}\label{eq7}
				\lim_{x\searrow 0} \eta_\lambda(x)\,=\,
				\begin{cases*}
					\frac{\Gamma(\alpha+1)\Gamma(-\beta)}{\Gamma\left(\frac{\alpha-\beta+1}{2}+\gamma\right)\Gamma\left(\frac{\alpha-\beta+1}{2}-\gamma\right)}
					,&$\beta<0,\;\alpha >-1$,
					\\
					\frac{\Gamma(1-\alpha)\Gamma(-\beta)}{\Gamma\left(\frac{-\alpha-\beta+1}{2}+\gamma\right)\Gamma\left(\frac{-\alpha-\beta+1}{2}-\gamma\right)}
					,&$\beta<0,\;\alpha \le -1$,
					\\
					\infty,&$\beta \ge 0$
				\end{cases*}
			\end{equation*}
			and the above limit is positive if it is finite.
			\item It holds that
			\begin{equation*}\label{eq27}
				\lim_{x\nearrow d} \,\eta_\lambda'cm(x)\,=\,
				\begin{cases*}
					0,&$\alpha>-1$,
					\\
					\frac{\alpha \sigma^2}{2},&$\alpha \le -1$.
				\end{cases*}
			\end{equation*}
			\item Under the additional assumption
			\[\lambda \,> \,\frac{\sigma^2}{2}
			\left(
			\frac{\alpha+\beta+1}{2}
			\right)^2\]
			for $\alpha\le -1$ and $\beta<-1$ it holds that
			\begin{align*}\label{eq50}&
				\lim_{x\searrow 0} \,\eta_\lambda'cm(x)\\=\,&
				\begin{cases*}
					\frac{-\lambda\Gamma(\alpha+1)\Gamma(\beta+1)}{\Gamma\left(
						\frac{\alpha+\beta+3}{2}+\gamma
						\right)
						\Gamma\left(
						\frac{\alpha+\beta+3}{2}-\gamma
						\right)},&$\beta >-1,\;\alpha>-1$,
					\\\left[\frac{2\lambda}{\sigma^2}-\alpha(\beta+1)\right]\frac{-\sigma^2\Gamma(1-\alpha)\Gamma(\beta+1)}{2\Gamma\left(\frac{-\alpha+\beta+3}{2}+\gamma\right)\Gamma\left(\frac{-\alpha+\beta+3}{2}-\gamma\right)},&$\beta >-1,\;\alpha\le-1$,\\
					-\infty,&$\beta \le-1$
				\end{cases*}
			\end{align*}
			and the above limit is negative if it is finite.
		\end{enumerate}
	\end{cor}
	
	As a consequence of the monotonicity properties of $\xi_\lambda$ and $\eta_\lambda$  under the additional condition in Proposition \ref{mon},  the two functions are in particular linearly independent. This allows us to make the following observation on $\FF_\lambda^\perp$, which follows along the lines of basic theory of elliptic partial differential equations in Hilbert spaces.
	\begin{lemma}\label{lemma1}If $\alpha\le -1$ or $\beta\le -1$ we additionally assume that
		\begin{equation}\label{eq186}
			\lambda \,> \,\frac{\sigma^2}{2}
			\left(
			\frac{\alpha+\beta+1}{2}
			\right)^2,
		\end{equation}
		then we have $\FF^\perp_\lambda\subset \spa\{\xi_\lambda,\eta_\lambda\}$.
	\end{lemma}
	\begin{proof}Let $0<l<r<d$, then 
		\begin{equation}\label{eq181}
			(f,g)\mapsto\int_{l}^{r} f'g' \,dcm\,+\,\lambda\int_{l}^{r} fg\, dm
		\end{equation}
		defines an inner product on $H^1((l,r))$ which is equivalent to $(\cdot,\cdot )_{H^1((l,r))}$. We denote the $2$-dimensional orthocomplement of $H_0^1((l,r))$ in $H^1((l,r))$ with respect to \eqref{eq181} by $C_\lambda((l,r))$. Now, if we have $f\in\mathcal{F}_\lambda^\perp$, then in particular
		\[
		\forall \varphi\in C_c^\infty((l,r)):\;\;\int_l^r \varphi' f'\, dcm\,+\,
		\lambda\int_l^r\varphi f \,dm\,=\,0,
		\] such that $f\in C_\lambda((l,r))$. Hence, it suffices to verify $C_\lambda = \spa\{\xi_\lambda,\eta_\lambda\}$ for which we first observe that $\xi_\lambda,\eta_\lambda\in C^1([l,r])\subset H^1((l,r))$. The identity \eqref{eq37} yields that
		\begin{gather}\label{eq2}
			(\xi_\lambda'cm)'(x)\,=\,c(x)m(x)\xi_\lambda''(x)+(a-bx)m(x)\xi_\lambda'(x)\,=\,\lambda m(x)\xi_\lambda(x)
		\end{gather}
		for all $x\in (0,d)$. Integration by parts gives us
		\[
		\forall \varphi\in C_c^\infty((l,r)):\;\;\int_l^r \varphi' \xi_\lambda'cm(x)\, dx\,+\,
		\lambda\int_l^r\varphi \xi_\lambda m(x) \,dx=0,
		\] 
		such that indeed  $\xi_\lambda\in C_\lambda((l,r))$. Analogously we  conclude $\eta_\lambda\in C_\lambda((l,r))$  and the claim follows by Proposition  \ref{mon} and Corollary \ref{f2} (ii).
	\end{proof}
A more detailed analysis yields a full characterization of $\FF_\lambda^\perp$.
	\begin{thm}\label{ortho} The space $\FF_\lambda^\perp$ is given by the following  table depending on $\alpha$ and $\beta$.
		\begin{equation}\label{cases_comp}
			\begin{array}{|c|c|c|c|}\hline
				&\alpha\le -1&-1<\alpha<0&\alpha\ge 0\\\hline \beta\le -1&\{0\}&\spa \{
				\xi_\lambda
				\}&\{0\}\\\hline 
				-1<\beta<0&\spa \{
				\eta_\lambda
				\}&\spa \{
				\xi_\lambda,\eta_\lambda
				\}&\spa \{\eta_\lambda
				\}\\\hline 
				\beta\ge 0&\{0\}&\spa \{
				\xi_\lambda
				\}&\{0\}\\\hline 
			\end{array}
		\end{equation}
		
	\end{thm}
	\begin{proof}We denote the space \eqref{cases_comp} according to $\alpha$ and $\beta$ by $\mathcal{G}_\lambda$ and verify first  that $\mathcal{G}_\lambda\subset\mathcal{F}_\lambda^\perp$. 
		For this purpose we note that integration by parts together with \eqref{eq2}  yields
		\begin{gather*}
			\int_\epsilon^{d-\epsilon} \xi_\lambda'\xi_\lambda'cm(x)\,dx\,+\,\lambda\int_\epsilon^{d-\epsilon}\xi_\lambda\xi_\lambda m(x)\,dx \,=\, \big[
			\xi_\lambda \xi_\lambda'cm
			\big]_\epsilon^{d-\epsilon}
		\end{gather*}
		for $\epsilon >0$.
		By letting  $\epsilon\searrow 0$ we get
		\begin{equation}\label{eq222}
			\int_X \xi_\lambda'\xi_\lambda'\,dcm(x)\,+\,\lambda\int_X\xi_\lambda\xi_\lambda \,dm(x) \,=\, \lim_{\epsilon\searrow 0} \,\big[
			\xi_\lambda \xi_\lambda'cm
			\big]_\epsilon^{d-\epsilon}.
		\end{equation}
		Part \ref{l_boundary} and \ref{l_boundary2} of Lemma \ref{prop_f1} imply that $
		\xi_\lambda \xi_\lambda'cm(x)$  converges as $x\searrow 0$. If we assume $-1<\alpha <0$, part \ref{r_boundary} and \ref{r_boundary2} yield that  $\xi_\lambda \xi_\lambda'cm(x)$ converges as well as $x\nearrow d$. In particular, \eqref{eq222} is finite in this case and $\xi_\lambda\in D(\mathcal{E})$. By the same integration by parts argument as in the proof of Lemma \ref{lemma1} we conclude that $\xi_\lambda$ is  orthogonal to $\mathcal{F}$ in $(D(\mathcal{E}),\mathcal{E}_\lambda)$.
		It follows  $\xi_\lambda \in\mathcal{F}^\perp_\lambda$ if $-1<\alpha<0$. Analogously one can show that $\eta_\lambda\in\mathcal{F}_\lambda^\perp$ for $-1<\beta<0$ such that indeed $\mathcal{G}_\lambda\subset \mathcal{F}^\perp_\lambda$. 
		
		To prove also the reverse inclusion $\mathcal{F}_\lambda^\perp \subset \mathcal{G}_\lambda$  we first impose the additional assumption 
		\begin{equation}\label{eq188}\lambda \,> \,\frac{\sigma^2}{2}
			\left(
			\frac{\alpha+\beta+1}{2}
			\right)^2\end{equation}
		whenever $\alpha\le -1$ or $\beta\le -1$.
		Then any $f\in\mathcal{F}_\lambda^\perp$ is of the form $f=c_\xi\xi_\lambda+c_\eta\eta_\lambda$ for some $c_\xi,c_\eta\in\mathbb{R}$ due to Lemma \ref{lemma1}.
		In particular, we can identify $f$  with this $dm$-version which is infinitely often differentiable on $(0,d)$ and satisfies $Gf=\lambda f$.
		As in the proof of Lemma \ref{lemma1} we can conclude $(f'cm)'(x)= \lambda m(x)f(x)$ for all $x\in (0,d)$. Integration by parts and the monotone convergence theorem give us that
		\begin{align*}
			\begin{split}\label{eq1}
				\big[ff'cm\big]_{\frac{d}{2}}^{d-\epsilon}\,=\,\int_{\frac{d}{2}}^{d-\epsilon} f'f'cm(x)+\lambda ffm(x)\,dx\,\to\, \,\int_{\frac{d}{2}}^{d} f'f'cm(x)+\lambda ffm(x)\,dx\,<\, \infty
			\end{split}
		\end{align*}
		as $\epsilon\searrow 0$. Consequently,
		$(ff'cm)(x)$ converges as $x\nearrow d$. 
		We assume that $\alpha\le -1$ and $c_\xi\ne 0$ and lead this to a contradiction. To this end we use the decomposition
		\begin{align*}
			(ff'cm)(x)\,=\, c_\xi\xi_\lambda'cm(x)\big[c_\xi\xi_\lambda(x)+c_\eta\eta_\lambda(x)\big]\,+\,	  c_\eta\eta_\lambda'cm(x)\big[c_\xi\xi_\lambda(x)+c_\eta\eta_\lambda(x)\big].
		\end{align*}
		Lemma \ref{prop_f1} \ref{r_boundary} and Corollary \ref{f2} (iii) yield that there exists a positive real number $y$ such that
		\[
		\xi_\lambda(x)+\frac{c_\eta}{c_\xi}\eta_\lambda(x)\,\to \, y
		\]
		as $x\nearrow d$. Corollary \ref{f2}  (v) gives us that $\eta_\lambda'cm(x)$ converges as $x\nearrow d$. Lemma \ref{prop_f1} \ref{r_boundary2} implies that $\xi_\lambda'cm(x)$ goes to infinity as $x\nearrow d$. All in all we get that
		\[
		(ff'cm)(x)\to \infty
		\] for $x\nearrow d$
		which is a contradiction. Next, we assume $\alpha \ge 0$ and  $c_\xi\ne 0$.
		Then we split the term $ff'cm(x)$ into the parts 
		\[ff'cm(x)\,=\,
		c_\xi\xi_\lambda(x)\big[c_\xi\xi_\lambda'cm(x)+c_\eta\eta_\lambda'cm(x)\big]\,+\,	c_\eta\eta_\lambda(x)\big[c_\xi\xi_\lambda'cm(x)+c_\eta\eta_\lambda'cm(x)\big].
		\]
		Using Lemma \ref{prop_f1} \ref{r_boundary2} and Corollary \ref{f2} (v) we  conclude that there exists a positive real number $y$ such that
		\[
		\xi_\lambda'cm(x)+\frac{c_\eta}{c_\xi}\eta_\lambda'cm(x)\,\to\, y
		\]
		for $x\nearrow d$. Corollary \ref{f2} (iii) yields that $\eta_\lambda(x)$ converges as $x\nearrow d$. By Lemma \ref{prop_f1} \ref{r_boundary} the term   $\xi_\lambda(x)$ approaches infinity as $x\nearrow d$ and we again get the contradiction $	
		\lim_{x\nearrow d}(ff'mc)(x)= \infty
		$.
		Analogously we can conclude that $\beta \le -1$ or $ \beta \ge 0 $ implies $c_\eta=0$.
		
		This proves $\mathcal{F}_\lambda^\perp=\mathcal{G}_\lambda$ under the additional assumption \eqref{eq188} for $\alpha \le -1$ or $\beta\le -1$.
		When $\alpha\le-1$ or $\beta\le -1$ and \eqref{eq188} is not satisfied we can consider a sufficiently large $\lambda^\dag$ and apply the just proven statement to conclude that
		$\codim (\mathcal{F})=\dim (\mathcal{G}_\lambda)$. Then the claim follows since we have already shown $\mathcal{G}_\lambda\subset \mathcal{F}_\lambda^\perp$. 
	\end{proof}
	
	The above statement allows us to characterize the space $\FF$ using that $\FF=\left(\FF_\lambda^\perp\right)^\perp$ in $(D(\EE), \EE_\lambda)$.
	
	\begin{thm}\label{DE}
		The space $\mathcal{F}$ is given by
		\begin{align}\begin{split}
				\label{Char}\Big\{
				f\in\mathcal{D(\mathcal{E})}:\, &\lim_{x\nearrow d} f^*(x)=0 \text{ if }-1<\alpha<0\\
				\wedge&\lim_{x\searrow 0} f^*(x)=0 \text{ if }-1<\beta<0\Big\}.
			\end{split}
		\end{align}
	\end{thm}

	\begin{proof}
		We denote the subspace from \eqref{Char} by $\mathcal{H}$ and verify first the inclusion $\mathcal{F}\subset \mathcal{H}$. 
		Let $f\in \mathcal{F}$ and assume $-1<\alpha<0$. Then $f$ is  orthogonal to $\xi_\lambda$ in $(D(\mathcal{E}),\mathcal{E}_\lambda)$ by Theorem \ref{ortho}. Using  \eqref{int_by_parts} together with \eqref{eq2}, the dominated convergence theorem implies that
		\[
		\big[f^*\xi_\lambda'cm\big]_\epsilon^{d-\epsilon}\,=\,
		\int_\epsilon^{d-\epsilon}
		f'\xi_\lambda'cm(x)+\lambda f\xi_\lambda  m(x) \,dx\,\to\,0
		\]
		as $\epsilon\searrow 0$.
		We note that $\xi_\lambda' cm(x)$ tends  towards a positive number for $x\nearrow d$ by Lemma \ref{prop_f1} \ref{r_boundary2}. Hence if we can show 
		\begin{equation}\label{eq332}\lim_{x\searrow 0}\,f^*\xi_\lambda'cm(x)\,= \,0,\end{equation} it follows that $\lim_{x\nearrow d} f^*(x)=0$. To this end we distinguish different cases.
		First, we assume $\beta >-1$. Then Lemma \ref{embedding} yields
		$
		\lim_{x\searrow 0} f^*cm(x)=0
		$. Furthermore, $\xi_\lambda'$ is  a hypergeometric function by termwise differentiation and thus admits a finite limit at $0$. Hence the desired equality follows.
		We consider next  $\beta \le -1$, then Lemma \ref{embedding} yields  $\lim_{x\searrow 0} f^*(x)=0$. Lemma \ref{prop_f1} \ref{l_boundary2} gives us that $\xi_\lambda'cm(x)$ converges to a real number as $x\searrow 0$ such that  \eqref{eq332} holds true again.
		Analogously one can show that $-1<\beta<0$ implies $\lim_{x\searrow 0} f^*(x)=0$ such that indeed $\mathcal{F}\subset \mathcal{H}$.
		
		To prove also $\mathcal{H}\subset \mathcal{F}$ we distinguish  different cases of $\alpha$ and $\beta$. Assume first 
		\[(
		\alpha\le-1 \vee \alpha\ge 0) \,\wedge\, (
		\beta\le-1 \vee \beta\ge 0).
		\]
		In this case Theorem \ref{ortho} yields $\mathcal{F}=D(\mathcal{E})$ and thus the inclusion is trivial. For  the remaining cases we take some $h\in\mathcal{H}$. Then there exists an $f\in \mathcal{F}$ and $g\in\mathcal{F}_\lambda^\perp$ such that $h=f+g$.
		If we assume
		\[(-1<\alpha<0)\, \wedge\, (
		\beta\le-1 \vee \beta\ge 0),\]
		there exists a $c_\xi\in\mathbb{R}$ such that $g=c_\xi\xi_\lambda$ due to Theorem \ref{ortho}.
		Furthermore, since  $\mathcal{F}\subset \mathcal{H}$ it follows $\lim_{x\nearrow d} f^*(x)=0$ and since $\lim_{x\nearrow d} h^*(x)=0$ we deduce that also $\lim_{x\nearrow d} c_\xi\xi_\lambda(x)=0$. By Lemma \ref{prop_f1} \ref{r_boundary} we have $\lim_{x\nearrow d} \xi_\lambda(x)> 0$ and thus we  conclude $c_\xi=0$. In particular, it holds $h=f\in\mathcal{F}$. The case
		\[ (
		\alpha\le-1 \vee \alpha\ge 0)\,\wedge\,(-1<\beta<0) \]
		can be treated analogously.	
		It remains the case
		\[(-1<\alpha<0) \,\wedge\, (-1<\beta<0).
		\]
		Then there exist $c_\xi,c_\eta\in\mathbb{R}$ such that $g=c_\xi\xi_\lambda+c_\eta\eta_\lambda$ by Theorem \ref{ortho}. As before it follows that
		\begin{equation}\label{eq636}
			\lim_{x\searrow 0} c_\xi\xi_\lambda(x)+c_\eta\eta_\lambda(x)\,=\,
			\lim_{x\nearrow d} c_\xi\xi_\lambda(x)+c_\eta\eta_\lambda(x)\,=\,0.
		\end{equation}
		Lemma \ref{prop_f1} \ref{l_boundary} and Corollary \ref{f2} (iv) yield that the limits $\lim_{x\searrow 0} \xi_\lambda(x)$ and $\lim_{x\searrow 0} \eta_\lambda(x)$ exist and are positive. Similarly, Lemma \ref{prop_f1} \ref{r_boundary} and Corollary \ref{f2} (iii) imply that the limits $\lim_{x\nearrow d} \xi_\lambda(x)$ and $\lim_{x\nearrow d} \eta_\lambda(x)$ exist and are positive. Thus we can rewrite \eqref{eq636} as
		\begin{equation}\label{mat_eq}
			\left(\begin{array}{cc}
				\lim_{x\searrow 0}\xi_\lambda(x)&\lim_{x\searrow 0}\eta_\lambda(x)\\\lim_{x\nearrow d}\xi_\lambda(x)&\lim_{x\nearrow d}\eta_\lambda(x)\\
			\end{array}\right)
			\left(\begin{array}{c}
				c_\xi\\c_\eta
			\end{array}\right)\, =\,\left(\begin{array}{c}
				0\\0
			\end{array}\right).
		\end{equation}
		Taking $\lambda$ sufficiently large together with Proposition \ref{mon} yields that $\xi_\lambda$  is strictly monotonically increasing and with Corollary \ref{f2} (ii) that  $\eta_\lambda$ is strictly monotonically decreasing. Hence the matrix from \eqref{mat_eq} is invertible and we conclude that $c_\xi=c_\eta=0$. It follows again  $h=f\in \FF$.
	\end{proof}
	As a consequence, we obtain the following observation on the boundary behaviour of functions from $D(\EE)$ in the fashion of Lemma \ref{embedding}.
	\begin{cor}\label{limiz}
		Let $f\in D(\mathcal{E})$. Then it holds
		\[
		\begin{cases}
			\lim_{x\nearrow d} f^*cm(x)\,=\,0, & \alpha\ge 0,\\
			\lim_{x\nearrow d} f^*(x)\,\in\,\mathbb{R}, & -1<\alpha<0,\\
			\lim_{x\nearrow d} f^*(x)\,=\,0, & \alpha\le -1,\\
		\end{cases}
		\]
		as well as
		\[
		\begin{cases}
			\lim_{x\searrow 0} f^*cm(x)\,=\,0, & \beta\ge 0,\\
			\lim_{x\searrow 0} f^*(x)\,\in\,\mathbb{R}, & -1<\beta <0,\\
			\lim_{x\searrow 0} f^*(x)\,=\,0, & \beta\le -1.\\
		\end{cases}
		\]	
	\end{cor}
	\begin{proof}We prove the first part of the statement, the second part can be shown analogously.
		Let $f\in D(\mathcal{E})$, then there exists a function $g\in \mathcal{F}$ and $c_\xi,c_\eta\in \mathbb{R}$ such that $f=g+c_\xi\xi_\lambda+c_\eta\eta_\lambda$ due to Theorem \ref{ortho}. We  distinguish between different cases of  $\alpha$.
		Assume $\alpha\le -1$. Then Lemma \ref{embedding} yields that $g^*(x)$ converges to $ 0$ as $x\nearrow d$ and Theorem \ref{ortho} implies $c_\xi=0$. Corollary \ref{f2} (iii) gives us that   $\eta_\lambda(x)$ converges to $0$ as $x\nearrow d$ as well and thus the claimed identity follows.
		Next, we consider $-1<\alpha <0$. Then $g^*(x)$ converges to $0$ as $x\nearrow d$ by Theorem \ref{DE}. Lemma \ref{prop_f1} \ref{r_boundary} and Corollary \ref{f2} (iii) state that $\xi_\lambda(x)$ and $\eta_\lambda(x)$ converge as $x\nearrow d$. Hence $f^*(x)$ converges as $x\nearrow d$. 
		Lastly, we consider $\alpha\ge 0$. Then Theorem \ref{ortho}  gives us $c_\xi=0$. Lemma \ref{embedding} yields that $g^*cm(x)$ converges to $0$ for $x\nearrow d$. By Corollary \ref{f2} (iii)  $\eta_\lambda(x)$ converges as $x\nearrow d$ and hence $\eta_\lambda cm(x)$ tends to $ 0$ for $x\nearrow d$. The claimed identity follows.
	\end{proof}
	
	\subsection{Sobolev type embeddings}
	In this section we analyze the embedding properties of $D(\EE)$ in spaces of continuous functions and spaces $L^q(X,dm)$ for $q>2$. To formulate it, we introduce the space $\tilde{X}$ as the interval $(0,d)$ including $\{d\}$ for $\alpha <0$ and  $\{0\}$ for $\beta <0$. We write $C(\tilde{X})$ for the space of continuous functions on $\tilde{X}$ and note that $C(\tilde{X})$ is a complete metric space equppied with the topology of uniform convergence on compact subsets, see \cite[pp.167-168]{folland1999real}.
		\begin{prop}\label{cont2}
		Every function from $D(\mathcal{E})$ has a $dm$-version $\tilde{f}$ which is continuous on $\tilde{X}$. In particular, the mapping
		\begin{equation}\label{emb_cont}
		D(\EE) \to C(\tilde{X}), f\mapsto \tilde{f}
		\end{equation}
		is a continuous embedding.
	\end{prop}
	\begin{proof} Let $f\in D(\mathcal{E})$, then  $f^*$ is by definition continuous on $(0,d)$. The  space $\tilde{X}$ contains the boundary point $d$ iff $\alpha <0$ and in this case Corollary \ref{limiz} yields  that $f^*$ admits a limit at $d$. Analogously, it follows  that $f^*$ has a limit at $0$ whenever $0\in\tilde{X}$. By redefining $f^*(0)$ and $f^*(d)$ as the corresponding limit we get the desired $dm$-version of $f$, which is necessarily unique. Consequently, the embedding \eqref{emb_cont} is well-defined. Since every $D(\EE)$-convergent sequence has a $dm$-almost everywhere convergent subsequence it follows that \eqref{emb_cont} has closed graph and therefore is continuous by the closed graph theorem, see \cite[Theorem 2.3, p.78]{axler1999topological}.
	\end{proof}
By the above observations, if $\alpha, \beta <0$, we have the embedding $D(\EE)\hookrightarrow L^\infty(X,dm)$. Since trivially  $D(\EE)\hookrightarrow L^2(X,dm)$ also the embeddings $D(\EE)\hookrightarrow L^q(X,dm)$ for $2\le q\le \infty$ follow by H\"older's inequality. To derive a similar result in the case that $\alpha\ge 0$ or $\beta\ge 0$ we state a special case of the characterization of embeddings of weighted Sobolev spaces from \cite{rabier2011embeddings}. To formulate it, we introduce for $s\in \RR$, $1\le q\le \infty$  the weighted spaces $L^q_s(\RR)$ as the $L^q$ space on $\RR$ equipped with the measure with density $|x|^s$ with respect to the Lebesgue measure. 
Moreover, we set
\begin{equation}\label{weighted_sob_space}
\|f\|_{W^{1,(2,2)}_{s, s+1}(\RR\setminus\{0\})}^2\,=\,
\|f\|_{L_s^2(\RR)}^2\,+\,\|f'\|_{L_{s+1}^2(\RR)}^2,
\end{equation}
for functions $ f$, which are  weakly differentiable  on $\RR\setminus \{0\}$ and define  $W^{1,(2,2)}_{s, s+1}(\RR\setminus\{0\})$ as the space of all $f$ such that \eqref{weighted_sob_space} is finite. As a consequence of the completeness of the involved spaces $L_s^2(\RR)$ and $L_{s+1}^2(\RR)$ the space $W^{1,(2,2)}_{s, s+1}(\RR\setminus\{0\})$ is complete as well.

\begin{lemma}\label{rabier}Let $s\ge 0$ and $1\le q<\infty$. If we define
	\[
		q_s=\begin{cases}
		2\left(1+\frac{1}{s}\right),&s>0,\\
		\infty,&s= 0,\\
	\end{cases}
	\]then $W^{1,(2,2)}_{s, s+1}(\RR\setminus\{0\})$ embeds continuously in $L^{q}_{s}(\RR)$ if and only if 
	$q\in [2,q_s]$.
\end{lemma}
\begin{proof}
	The stated embedding holds if and only if any of the cases (i)-(vi) of \cite[Theorem 1.1, p.253]{rabier2011embeddings} holds. The cases (ii), (v) and (vi) are impossible and the case (iii) corresponds to the trivial case $q=2$. The case (iv) is impossible if $s=0$, else it corresponds to $q=q_s$. Finally, we note that the case (i) is equivalent to $q\in (2,q_s)$.
\end{proof}
Using a cutoff argument we can transfer these findings to the setting of a bounded interval, which results in the following theorem. During its proof we denote by $K_{\dots}$ a constant only depending on its indices, which may change from line to line.
	\begin{thm}\label{sob_emb}
		We define $q_* = \min\{q_\alpha, q_\beta\}$, where\[
		q_\alpha=\begin{cases}
			2\left(1+\frac{1}{\alpha}\right),&\alpha>0,\\
			\infty,&\alpha\le 0,\\
		\end{cases}\;\; \text{ and }\;\;
		q_\beta=\begin{cases}2\left(1+
			\frac{1}{\beta}\right),&\beta>0,\\
			\infty,&\beta\le 0.\\
		\end{cases}
		\]
		Then there is a  continuous embedding $D(\EE)\hookrightarrow L^{q}(X,dm)$ for $q\in [2,q_*)$. If additionally $\alpha, \beta \ne 0$, then the embedding is also true for $q=q_*$.
	\end{thm}
	\begin{proof}
		Let $\alpha<0$. Then the restriction mapping
		\[
		C(\tilde{X})\to C\left(\left[\frac{d}{2},d\right]\right), f\mapsto f|_{\left[\frac{d}{2},d\right]}
		\]
		is continuous. Therefore, by Proposition \ref{cont2} we can estimate
		\[
		\|f\|_{L^\infty\left(\left[\frac{d}{2},d\right], dm\right)}\, =\,  
		\|\tilde{f}\|_{C\left(\left[\frac{d}{2},d\right]\right)}\, \le \,K_{\alpha,\beta,d} \sqrt{\EE_1(f,f)}
		\]
		for every $f\in D(\EE)$.
		Due to H\"older's inequality we conclude that
		\[
		\|f\|_{L^q\left(\left[\frac{d}{2},d\right], dm\right)}\,\le \, \|f\|_{L^2\left(\left[\frac{d}{2},d\right], dm\right)}^{\frac{2}{q}}\|f\|_{L^\infty\left(\left[\frac{d}{2},d\right], dm\right)}^{\frac{q-2}{q}} \le \,K_{\alpha,\beta,d} \sqrt{\EE_1(f,f)}.
		\]
		for any $q\in [2, q_\alpha]$.
		
		For $\alpha \ge 0$ we choose a smooth cutoff-function $\varphi \in C_c^\infty(\RR_{\ge 0})$, $0\le \varphi\le 1$ such that $\varphi(x)=1$ for $0\le x<\frac{d}{2}$ and  $\varphi(x)=0$ for $x>\frac{3d}{4}$. For $f\in D(\EE)$ we define the function 
		\[
		f_\varphi(x)\,=\, \begin{cases}
			f(d-x)\varphi(x), &x\in (0,d),\\0,&\text{else}.
		\end{cases}
		\]
		Then for any $1\le q<\infty$ we have the estimate
		\begin{align}\begin{split}
			\label{eq67}
			\|f\|_{L^q\left(\left[\frac{d}{2},d\right], dm\right)}\,=&\, K_{\alpha,\beta, d}\left(\int_\frac{d}{2}^d |f(x)|^q x^\beta (d-x)^\alpha\,dx\right)^\frac{1}{q}
			\\\le& \,K_{\alpha,\beta, d} \left(\int_0^\frac{d}{2} |f_\varphi(x)|^q x^\alpha\,dx\right)^\frac{1}{q}\,\le \, K_{\alpha,\beta, d}\|f_\varphi\|_{L^q_\alpha(\RR)}.
			\end{split}
		\end{align}Moreover, for any $\psi\in C_c^\infty(\RR\setminus \{0\})$ we have $\varphi\psi\in C_c^\infty((0,d))$, where we use the convention $\varphi(x)=0$ for $x<0$, and consequently
	\begin{align*}
		-&\int_{\RR} f_\varphi(x)\psi'(x) \, dx\,=\, 
		-\int_0^d f(d-x)\left[(\varphi\psi)'(x)-\varphi'(x)\psi(x)\right] \, dx\\=\,&
		\int_0^d f(d-x)\varphi'(x)\psi(x)\,dx\,-\, 
		\int_0^d f'(d-x)(\varphi\psi)(x)\,dx
		\\=\,&\int_{\RR}\left[
		f(d-x)\varphi'(x)-f'(d-x)\varphi(x)
		\right]\psi(x)\, dx.
	\end{align*}
Hence $f_\varphi$ is weakly differentiable on $\RR\setminus \{0\}$ with the $\psi$-independent part of the latter integrand as weak derivative. We can estimate
\begin{align*}
	\|f_\varphi\|_{L_\alpha^2(\RR)}^2\,=\, \int_0^d f^2(d-x)\varphi^2(x)x^\alpha\,dx
	\,\le\,\int_\frac{d}{4}^d f^2(x)({{d-x}})^\alpha\,dx
	\,\le\, K_{\alpha,\beta,d}\|f\|_{L^2(X,dm)}^2.
\end{align*}Similarly, we obtain that
\begin{align*}&
	\|f_\varphi'\|_{L_{\alpha+1}^2(\RR)}^2\,\le \, \int_0^d2 \left[f^2(d-x)\varphi'^2(x)
	+f'^2(d-x)\varphi^2(x)
	\right]x^{\alpha+1}\,dx
	\\\le\,&2\left[\sup_{y\in \RR} |\varphi'(y)|^2\int_{\frac{d}{4}}^\frac{d}{2}
	f^2(x) (d-x)^{\alpha+1} \,dx\,+\, 
	\int_{\frac{d}{4}}^d
	f'^2(x) (d-x)^{\alpha+1} \,dx
	\right]\\\le \,& K_{\alpha,\beta,d,\varphi} \|f\|_{L^2(X,dcm)}^2 \,+\, K_{\alpha,\beta,d} \|f'\|_{L^2(X,dcm)}^2 \,\le\,  K_{\alpha,\beta,d,\varphi}\EE_1(f,f).
\end{align*}
In the last line we used that the function $c$ is bounded. Combining now \eqref{eq67}, Lemma \ref{rabier}  and the previous two estimates, we obtain
that
\[
\|f\|_{L^q\left(\left[\frac{d}{2},d\right], dm\right)}
\,\le \,K_{\alpha,\beta,d}\|f_\varphi\|_{L^q_\alpha(\RR)}\,\le \, K_{\alpha,\beta,d,q}\|f_\varphi\|_{W^{1,(2,2)}_{\alpha, \alpha+1}(\RR\setminus\{0\})}\,\le \,
K_{\alpha,\beta,d,\varphi,q}\sqrt{\EE_1(f,f)}.
\]
for any finite $q\in [2, q_\alpha]$. Analogously, we find that 
\[
\|f\|_{L^q\left(\left[0,\frac{d}{2}\right], dm\right)}
\,\le \,
K_{\alpha,\beta,d,\varphi,q}\sqrt{\EE_1(f,f)}
\]
for $q\in [2,q_\beta]$, where we additionally require $q$ to be finite for $\beta = 0$. The claim follows now by 
\[
\|f\|_{L^q(X,dm)}\,\le \,
\|f\|_{L^q\left(\left[0,\frac{d}{2}\right], dm\right)}\,+\, 
\|f\|_{L^q\left(\left[\frac{d}{2},d\right], dm\right)}.
\]
	\end{proof}
	\subsection{Regularity and quasi-notions}
	In this section, we use the findings from the previous ones to prove regularity of $(\EE,D(\EE))$ and determine what it means for a property to hold quasi-everywhere and a function to be quasi-continuous with respect to $(\EE,D(\EE))$. For the definition of these properties, we refer to  \cite[Section 1.1, pp.3-6; Section 2.1, pp.66-76]{fukushima2011}. We recall that the state space $X$ was defined as the interval $(0,d)$ together with the boundary point $\{d\}$ iff $\alpha>-1$ and $\{0\}$ iff $\beta>-1$.
	\begin{prop}\label{is_reg}
		The Dirichlet form $(\mathcal{E},D(\mathcal{E}))$ is regular.
	\end{prop}
\begin{proof}
	We provide a core for $(\EE,D(\EE))$ depending on $\alpha, \beta$. For $\alpha, \beta\le -1$ we choose
	$\mathcal{C}= C_c^\infty((0,d))$. Then $\mathcal{C}$ is dense in $D(\EE)$ since $D(\EE)=\FF$ by Theorem \ref{ortho} and $\mathcal{C}$ is dense in $C_c((0,d))$ with respect to uniform convergence. For the case  that $\alpha>-1$ and $\beta \le -1$ we choose instead $\mathcal{C}=\spa \left\{C_c^\infty((0,d))\cup \{f\}\right\}$, where
	\begin{equation}\label{function_f}
	f\colon (0,d]\to \mathbb{R},\, x\mapsto 
	\begin{cases}
		0,& x\le \frac{d}{2},\\
		x-\frac{d}{2},& x>\frac{d}{2}.
	\end{cases}
	\end{equation}
	Then for every $h\in C_c((0,d])$ one has $(h-cf)(d)=0$ for a suitable constant $c\in \RR$ and therefore $h-cf$  can be approximated by functions from $C_c^\infty((0,d))$ with respect to uniform convergence. Therefore, $\mathcal{C}$ is dense in $C_c((0,d])$. If $-1<\alpha<0$ one can argue analogously using Corollary \ref{limiz} and Theorem \ref{DE} to conclude that $\mathcal{C}$ is dense in $D(\EE)$. If instead $\alpha\ge 0$, then we have again $D(\EE)=\FF$ and therefore $C_c^\infty((0,d))$ is even dense in $D(\EE)$. The case $\alpha \le -1$ and $\beta > -1$ can be treated analogously by choosing $\mathcal{C}=\spa \left\{C_c^\infty((0,d))\cup \{g\}\right\}$ with \begin{equation}\label{function_g}
	g\colon [0,d)\to \mathbb{R},\, x\mapsto 
	\begin{cases}
		\frac{d}{2}-x,& x< \frac{d}{2},\\
		0,& x\ge\frac{d}{2}.
	\end{cases}
	\end{equation}
	For the remaining case $\alpha,\beta >-1$ one can choose $\mathcal{C}= \spa\left\{C_c^\infty((0,d))\cup \{f,g\}\right\}$ and argue as before.
\end{proof}

We denote the capacity of a subset  $B\subset X$ with respect to $(\EE,D(\EE))$ by $\Ca(B)$. Moreover, for $f\in D(\EE)$, we write $\tilde{f}$ for its $dm$-version in $C(\tilde{X})$, which exists by Proposition \ref{cont2}. In light of the next Theorem this is consistent with the fact that the same notation is frequently used for a quasi-continuous $dm$-version of $f$.
\begin{thm}\label{quasi_notions}The following holds.\begin{enumerate}[label=(\roman*)]
		\item A subset $B\subset X$ fulfills $\Ca(B)=0$ iff $B\subset X\setminus \tilde{X}$.
		\item A function $f\colon X \to \mathbb{R}\cup \{-\infty, \infty\}$ is quasi-continuous iff $f$ is real-valued and continuous on $X\cap \tilde{X}$.
	\end{enumerate}
\end{thm}
\begin{proof}
	We start by verifying part (i) and assume that $\alpha \ge 0$.  We consider again the function \eqref{function_f}, which is an element of $D(\EE)$. By Theorem \ref{DE}  there exists a sequence $(f_n)_{n\in \mathbb{N}}\subset C_c^\infty((0,d))$  such that $f_n\to f$ in $D(\EE)$ and by \cite[Theorem 2.1.4 (i), p.72]{fukushima2011} we have $f_{n_k}\to f$ quasi everywhere for a subsequence. Since $f_n(d) \nrightarrow f(d)$ the set $\{d\}$ has necessarily capacity $0$.  Similarly one shows that $\Ca (\{0\})=0$ whenever $\beta \ge 0$. Therefore, any $B\subset X\setminus \tilde{X}$ satisfies $\Ca(B)=0$. Conversely, we assume that there is an $x\in B\cap \tilde{X} $. If we can show that $\Ca(\{x\})>0$, the proof of (i) will be complete. By the definition of the capactiy and \cite[Lemma 2.1.1 (i), p.67]{fukushima2011}, there is a sequence of open sets $\mathcal{O}_n\subset X$ and functions $f_n\in D(\EE)$ such that $x\in\mathcal{O}_n$, $f_n\ge 1$ $dm$-almost everywhere on $\mathcal{O}_n$ for every $n\in \mathbb{N}$ and 
	$\EE_1(f_n, f_n)\to \Ca(\{x\})$. By Proposition \ref{cont2} the mapping 
	\[\delta_x\colon D(\EE)\mapsto \RR, f\mapsto \tilde{f}(x)\]
	is continuous and consequently we have \[
	\|\delta_x\|^2\EE_1(f_n, f_n) \ge |\tilde{f_n}(x)|^2\ge 1
	\]
	for every $n$. We conclude that the sequence $\mathcal{E}_1(f_n, f_n)$ is bounded away from $0$ which finishes the proof of (i).
	
	For part (ii) we notice first that if a function is real-valued and continuous on $X\cap \tilde{X}$, then it is also quasi-continuous as a consequence of (i). We assume conversely that $f$ is quasi-continuous.  Let $K_n$ be an increasing  sequence of compact subintervals of $X\cap \tilde{X}$ such that $\bigcup_{n\in \mathbb{N}} K_n =X\cap \tilde{X}$. It suffices to show that $f$ is continuous on each of the intervals $K_n$. 
		We fix an $n\in \mathbb{N}$ and observe that by Proposition \ref{cont2} we have 
		\[
		D(\EE) {\hookrightarrow} C(\tilde{X}) \overset{r}{\to} C(K_n),
		\]
		where $r$ denotes the restriction mapping. We obtain the estimate $\|\delta_x\| \le \|r\circ i\|$ for every $x\in K_n$ and following the steps of the proof of (i) we conclude that $\Ca(\{x\})\ge \frac{1}{\|r\circ i\|^2}$. By the definition of quasi-continuity, there exists an open set $\mathcal{O}\subset X$ such that $f$ is real-valued and continuous on $X\setminus \mathcal{O}$ and $\Ca(\mathcal{O})< \frac{1}{\|r\circ i\|^2}$. In particular $\mathcal{O}\cap K_n=\emptyset$, which completes the proof.
\end{proof}
\subsection{Hardy   type inequalities}
In this last subsection we provide  Hardy-type inequalities  for the integrability pairs $(p,q)=(2,1)$ and $(2,2)$. The former will give rise to reference functions for $(\EE, D(\EE))$ and the latter to an estimate on the spectral gap of $(L,D(L))$ later on. The proofs rely  on verifying conditions which are derived similarly to the more general situations \cite[Theorem 1, p.50]{Mazja_2011} and \cite[Theorem 1, p.733]{Persson_Stepanov2002}.
\begin{lemma}\label{lemma_ref_fct}If $\alpha \le -1$ or $\beta \le -1$ and  $r,s\in \RR$ such that either
	\begin{enumerate}[label=(\roman*)]\item $\beta \le -1$, $\alpha >-1$, $r>\frac{-2-\beta}{2}$ and $s>\max\{-1-\alpha,\frac{-2-\alpha}{2}\}$,
		\item $\beta> -1$, $\alpha \le-1$, $r>\max\{-1-\beta,\frac{-2-\beta}{2}\}$ and $s>\frac{-2-\alpha}{2}$ or
		\item $\beta \le -1$, $\alpha \le -1$, $r>\frac{-2-\beta}{2}$ and $s>\frac{-2-\alpha}{2}$.
	\end{enumerate}Then there is a constant $C_{\alpha, \beta, \sigma, d, r,s}<\infty$ such that
	\begin{equation}\label{Eq_Hardy}
		\int_X |f(x)| x^r(d-x)^s\, dm(x)
		\, \le \,	C_{\alpha, \beta, \sigma, d, r,s} \sqrt
		{\EE(f,f)}	\end{equation}
	for every $f\in D(\EE)$.
\end{lemma}
\begin{proof}Let $f\in D(\EE)$.
	In the case of (i) we have $\tilde{f}(0)=0$ by Corollary \ref{limiz}  such that 
	an application of \cite[Theorem 8.2, p.204]{brezis2010functional} and Fubini's theorem yields
	\begin{align*}&\int_0^{d} |f(x)| x^r(d-x)^s\,dm(x) \, \le \, 
		\int_X \int_0^x|f'(y)|\, dy\, \frac{x^{\beta+r}(d-x)^{\alpha+s}}{d^{\alpha+\beta+1}}\, dx\\= \, &
		\int_0^{d} \int_{y}^{d} \frac{x^{\beta+r}(d-x)^{\alpha+s}}{d^{\alpha+\beta+1}}
		\, dx \,|f'(y)|\,dy.
	\end{align*}
	If we can show that 
	\begin{equation}\label{Eq_Int}
		\int_0^dg(y)^2\frac{2d^{\alpha+\beta+1}}{\sigma^2y^{\beta+1}(d-y)^{\alpha+1}}\, dy
	\end{equation}
	is finite, where
	\[
	g(y)\,=\,\int_{y}^{d} \frac{x^{\beta+r}(d-x)^{\alpha+s}}{d^{\alpha+\beta+1}},
	\, dx
	\]
	an application of H\"older's inequality will yield that
	\begin{align*}&
		\int_X |f(x)| x^r(d-x)^s\, dm(x) \\ \le \,& \left(
		\int_0^dg(y)^2\frac{2d^{\alpha+\beta+1}}{\sigma^2y^{\beta+1}(d-y)^{\alpha+1}}\, dy\right)^\frac{1}{2}\left(
		\int_0^{d} |f'(y)|^2\frac{\sigma^2y^{\beta+1}(d-y)^{\alpha+1}}{2d^{\alpha+\beta+1}}\,dy\right)^\frac{1}{2}
	\end{align*}
	and the claim follows. To estimate \eqref{Eq_Int} we observe that for $y\ge \frac{d}{2}$
	\begin{align*}
		g(y)\, \le \, &C_{\alpha, \beta, d, r}\int _y^d  (d-x)^{\alpha+s}\, dx\,=\, C_{\alpha, \beta,  d,r}
		\int_0^{d-y} x^{\alpha+s}\, dx
		\,=\, C_{\alpha, \beta, d, r,s} (d-y)^{\alpha+s+1}
	\end{align*} 
	by the assumption $s>-1-\alpha$.
	Using additionally $s>\frac{-2-\alpha}{2}$ we conclude that
	\begin{equation}\label{Eq_Int1}
		\int_\frac{d}{2}^dg(y)^2\frac{2d^{\alpha+\beta+1}}{\sigma^2y^{\beta+1}(d-y)^{\alpha+1}}\, dy\, \le \,
		C_{\alpha, \beta, \sigma, d, r,s} 
		\int_\frac{d}{2}^d \frac{(d-y)^{2\alpha+2s+2} }{(d-y)^{\alpha+1}}\, dy\,=\,  C_{\alpha, \beta,  \sigma,d, r,s} \,<\,\infty.
	\end{equation}
	For $y<\frac{d}{2}$ we obtain instead the bound
	\begin{equation}\label{Eq1}
		g(y)\, \le  \,g\left(\frac{d}{2}\right)\,+\, C_{\alpha, \beta, d, s}\int_y^{\frac{d}{2}} x^{\beta+r}\, dx\, \le \, C_{\alpha, \beta, d, r,s}\begin{cases}(1+y^{\beta+r+1}), &\beta+r \le -1,\\(1+|\log(y)|),
			&
			\beta+r = -1,\\1,& \beta+r > -1.
		\end{cases}
	\end{equation}
	Therefore, by Young's inequality and the assumptions $\beta \le -1$ and $r>\frac{-2-\beta}{2}$
	\begin{align}\begin{split}
			\label{Eq_Int2}&
			\int_0^\frac{d}{2} g(y)^2\frac{2d^{\alpha+\beta+1}}{\sigma^2y^{\beta+1}(d-y)^{\alpha+1}}\, dy\, \le \,
			C_{\alpha, \beta, \sigma, d, r,s} \begin{cases}\int_0^\frac{d}{2} \frac{1+y^{2\beta+2r+2}}{y^{\beta+1}}\, dy, &\beta+r \le -1,\\\int_0^\frac{d}{2} \frac{1+\log(y)^2}{y^{\beta+1}}\, dy,
				&
				\beta+r = -1,\\\int_0^\frac{d}{2} \frac{1}{y^{\beta+1}}\, dy,& \beta+r > -1.
			\end{cases} 
			\\=\,& C_{\alpha, \beta, \sigma, d, r,s} \,<\,\infty.
	\end{split}\end{align} Adding up \eqref{Eq_Int1} and \eqref{Eq_Int2} we conclude that \eqref{Eq_Int} is indeed finite. The case (ii) can be treated analogously. For the last case (iii) we use that $\tilde{f}(d)=\tilde{f}(0)=0$ and obtain that
	\begin{align}\begin{split}\label{Eq_2}
			&\int_X |f(x)| x^r(d-x)^s\,dm(x) \\ \le \, & \int_0^{\frac{d}{2}}\int_0^x |f'(y)|\,dy\, \frac{x^{\beta+r}(d-x)^{\alpha+s}}{d^{\alpha+\beta+1}}\, dx\,+\,
			\int_{\frac{d}{2}}^d \int_x^d |f'(y)|\,dy\, \frac{x^{\beta+r}(d-x)^{\alpha+s}}{d^{\alpha+\beta+1}}.
		\end{split}
	\end{align}
	As before, we estimate the first integral of the right-hand side by 
	\begin{align*}&
		\int_0^{\frac{d}{2}}\int_y^\frac{d}{2} \frac{x^{\beta+r}(d-x)^{\alpha+s}}{d^{\alpha+\beta+1}}\, dx \,|f'(y)|\,dy \\\le
		\,&
		\left(\int_0^\frac{d}{2}h(y)^2\frac{2d^{\alpha+\beta+1}}{\sigma^2y^{\beta+1}(d-y)^{\alpha+1}}\, dy\right)^\frac{1}{2}
		\left(
		\int_0^{\frac{d}{2}} |f'(y)|^2 \frac{\sigma^2y^{\beta+1}(d-y)^{\alpha+1}}{2d^{\alpha+\beta+1}}\,dy\right)^\frac{1}{2},
	\end{align*}
	where we set this time 
	\[
	h(y)\,=\, \int_y^\frac{d}{2} \frac{x^{\beta+r}(d-x)^{\alpha+s}}{d^{\alpha+\beta+1}}\, dx.
	\]
	We estimate the function $h$ as before by
	\begin{align*}
		h(y)\, \le \,C_{\alpha, \beta, d, s} \int_y^\frac{d}{2}x^{\beta+r}\, dx,
	\end{align*}which is a term appearing in
	\eqref{Eq1}. In particular, continuing as in \eqref{Eq_Int2} we conclude finiteness of
	\[\int_0^\frac{d}{2}h(y)^2\frac{2d^{\alpha+\beta+1}}{\sigma^2y^{\beta+1}(d-y)^{\alpha+1}}\, dy
	\]
	due to $\beta \le -1$ and $r>\frac{-2-\beta}{2}$. The second term of the right-hand side of \eqref{Eq_2} can be estimated analogously and therefore the claim follows.
\end{proof}
\begin{lemma}\label{Lemma_Poinc}The following holds if either $\alpha\le -1$, $\beta>-1$ or $\alpha>-1$, $\beta\le -1$.
	\begin{equation}\label{Eq_10}
	\int_{X}f^2(x)\, dm(x)\, \le \,
	\EE(f,f)\cdot
	\begin{cases}
		\frac{2}{\sigma^2, \min\{|\alpha+1|, |\beta+1|\}^2}, &\alpha, \beta\ne -1,\\
		\frac{8}{\sigma^2} \sup_{x\in (0,1)} x(1-x)\left[\frac{1}{\alpha+1}-\log(x)\right]^2, &
		\beta=-1
		,\\\frac{8}{\sigma^2} \sup_{x\in (0,1)} x(1-x)\left[\frac{1}{\beta+1}-\log(x)\right]^2, &
		\alpha=-1
		.
	\end{cases}
	\end{equation}
\end{lemma}
\begin{rem}It would be interesting to show an estimate \eqref{Eq_10} with a good constant  also for the case $\alpha, \beta \le -1$, which requires probably a different approach. While the calculation  in \eqref{Eq_9} is still possible in this case, the hypergeometric function from \eqref{Eq_8} is not well-defined if $-(\alpha+2)\in \mathbb{N} $. Even if it is well-defined an application of Euler's integral representation \eqref{Eq_11}  requires $\alpha+1>0$.
\end{rem}
\begin{proof}We assume that $\beta\le -1$, $\alpha>-1$ and calculate using \cite[Corollary 8.10, p.215]{brezis2010functional}
	\begin{align}\begin{split}\label{Eq_9}
			&
		\int_{X}f^2(x)\, dm(x)\,\le\, \int_0^d \int_0^x 2 |f(y)f'(y)|\,dy \,m(x)\, dx\\=\,&
		2\int_0^d \int_y^d m(x)\, dx \,|f(y)f'(y)|\,dy\\\le \,&	2\left(
	\int_0^d f'(y)^2 cm(y)\,dy\right)^\frac{1}{2}\left(
	\int_0^d f^2(y)\left(\int_y^d m(x)\, dx\right)^2 cm(y)^{-1}\,dy
	\right)^\frac{1}{2}\\\le \,&
	2\left(
	\int_0^d f'(y)^2 cm(y)\,dy\right)^\frac{1}{2}\left(
	\int_0^d f^2(y)m(y)\, dy\right)^\frac{1}{2} \sup_{y\in (0,d)}\left[ \left(\int_y^d m(x)\, dx\right)^2 cm(y)^{-1}m(y)^{-1}\right]^\frac{1}{2}.
		\end{split}
	\end{align}
	Therefore, we have
	\begin{align*}	\int_{X}f^2(x)\, dm(x)\, \le \, 4\sup_{y\in (0,d)}\left[ \left(\int_y^d m(x)\, dx\right)^2 cm(y)^{-1}m(y)^{-1} \right]\EE(f,f)
	\end{align*}
	and it is left to estimate the supremum in front of $\EE(f,f)$. 
If we define
\begin{equation}\label{Eq_8}
M(y)\,=\, \frac{\left(1-\frac{y}{d}\right)^{\alpha+1}}{(\alpha+1)} {}_2F_1\left(-\beta, \alpha+1; \alpha+2; 1-\frac{y}{d}\right),
\end{equation}
we see directly that $M(d)=0$ since $\alpha+1>0$. Moreover, term-wise differentiation yields
\begin{align}\label{Eq_1}
	M'(y)\,=\, -\frac{\left(1-\frac{y}{d}\right)^{\alpha}}{d}  {}_2F_1\left(-\beta, \alpha+1; \alpha+1; 1-\frac{y}{d}\right)
\end{align}
The appearing hypergeometric function is of the form 
\[
\sum_{n=0}^\infty \frac{(-\beta)_n}{n!}\left(1-\frac{y}{d}\right)^n\,=\, \left(\frac{y}{d}\right)^{\beta}
\]
by Taylor expansion of the function $(1-x)^\beta$ in $x=0$. Inserting this into \eqref{Eq_1} yields 
$
M'(y)= -m(y)
$
such that 
\[
M(y)\,=\, \int_y^d m(x)\, dx.
\]
To estimate $M(y)$, we use Euler's transformation formula and Euler's integral representation, see \cite[Theorem 2.2.1, p.65, Theorem 2.2.5, p.68]{andrews1999special} to obtain that
\begin{align}\begin{split}\label{Eq_11}
		&
	 {}_2F_1\left(-\beta, \alpha+1; \alpha+2; 1-\frac{y}{d}\right)\,=\, 
	 \left(\frac{y}{d}\right)^{\beta+1}
	 {}_2F_1\left(\alpha+\beta+2, 1; \alpha+2; 1-\frac{y}{d}\right)\\\,=\,&(\alpha+1)
	 \left(\frac{y}{d}\right)^{\beta+1}\int_0^1 (1-t)^\alpha \left(1- \left(1-\frac{y}{d}\right)t\right)^{-(\alpha+\beta+2)}\, dt.
	\end{split}
\end{align}
If $\alpha+\beta+2\ge  0$ the hypergeometric function 
\[
{}_2F_1\left(\alpha+\beta+2, 1; \alpha+2; x\right)
\]
is monotonously increasing in $x$  and we can estimate it under the additional assumption $\beta<-1$ by its limit
\[
\lim_{x\nearrow 1} 
{}_2F_1\left(\alpha+\beta+2, 1; \alpha+2; x\right)\,=\, \frac{\alpha+1}{-(\beta+1)},
\]
due to  \eqref{eq17}. For $\beta=-1$ we have 
\begin{align*}&{}_2F_1\left(-\beta, \alpha+1; \alpha+2; 1-\frac{y}{d}\right)\,=\, \sum_{n=0}^\infty \left(1-\frac{y}{d}\right)^n \frac{(\alpha+1)_n}{(\alpha+2)_n}\\\le \,& 
	1\,+\,(\alpha+1)\sum_{n=1}^\infty \left(1-\frac{y}{d}\right)^n \frac{1}{\alpha+n+1}\,\le \, 
	1\,+\,(\alpha+1)\sum_{n=1}^\infty \left(1-\frac{y}{d}\right)^n \frac{1}{n}\\=\, &
	1\,-\, (\alpha+1)\log\left(\frac{y}{d}\right).
\end{align*}
If $\alpha+\beta+2<  0$, which implies in particular $\beta<-1$, we can instead estimate
\begin{align*}
	\int_0^1 (1-t)^\alpha \left(1- \left(1-\frac{y}{d}\right)t\right)^{-(\alpha+\beta+2)}\, dt\, \le \, 
	\int_0^1 (1-t)^\alpha\, dt\,=\, \frac{1}{\alpha+1}
\end{align*}
All in all we obtained the estimate
\[{}_2F_1\left(-\beta, \alpha+1; \alpha+2; 1-\frac{y}{d}\right)\, \le \, 
\begin{cases}
	\left(\frac{y}{d}\right)^{\beta+1}
	,&\alpha+\beta+2<0,\\\frac{\alpha+1}{-(\beta+1)}
	\left(\frac{y}{d}\right)^{\beta+1} 
	,&\alpha+\beta+2\ge  0, \;  \beta\ne -1,\\
	1- (\alpha+1)\log\left(\frac{y}{d}\right),&\beta=-1.
\end{cases}
\]
This implies
\[M(y)\,\le \, 
\begin{cases}
	\frac{\left(\frac{y}{d}\right)^{\beta+1}\left(1-\frac{y}{d}\right)^{\alpha+1}}{\min\{a+1, -(\beta+1)\}}, & \beta\ne -1, \\
	\left(1-\frac{y}{d}\right)^{\alpha+1}\left[\frac{1}{\alpha+1}-\log \left(\frac{y}{d}\right)\right],&\beta=-1,
\end{cases}
\]
which in the  case $\beta \ne -1$  leads us to
\begin{align*}&
\frac{M(y)^2}{cm(y)m(y)}\, \le \, \frac{2d^{2(\alpha+\beta+1)}	\left(\frac{y}{d}\right)^{2\beta+2}\left(1-\frac{y}{d}\right)^{2\alpha+2}}{\sigma^2 y^{2\beta+1}(d-y)^{2\alpha+1}\min\{a+1, -(\beta+1)\}^2}\\=\, &
\frac{2	\left(\frac{y}{d}\right)\left(1-\frac{y}{d}\right)}{\sigma^2 \min\{a+1, -(\beta+1)\}^2}\, \le \, 
\frac{1}{2\sigma^2 \min\{a+1, -(\beta+1)\}^2}.
\end{align*}
If $\beta=-1$ we obtain instead
\begin{align*}
	\frac{M(y)^2}{cm(y)m(y)}\, \le \,\frac{2\left( \frac{y}{d}  \right)  \left(1-\frac{y}{d}\right) \left[\frac{1}{\alpha+1}-\log \left(\frac{y}{d}\right)\right]^2}{\sigma^2}\, \le \, 
	\frac{2}{\sigma^2} \sup_{x\in (0,1)} x(1-x)\left[\frac{1}{\alpha+1}-\log(x)\right]^2.
\end{align*}
The case $\alpha \le -1$, $\beta>-1$ can be treated analogously.
\end{proof}

	\section{The corresponding semigroup}\label{Sec_Semigroup}
	In this section we analyze the properties of the markovian semigroup of symmetric contractions on $L^2(X,dm)$ associated to $(\EE,D(\EE))$, which we denote by $(T_t)_{t>0}$. For details on the correspondence between operator semigroups and closed forms we refer to \cite[Section 1.3, pp. 16-25]{fukushima2011}. We start by determining for which  parameters the semigroup $(T_t)_{t>0}$ is conservative, for a definition of this property, see \cite[p.56]{fukushima2011}.
	We use the following criterion from
		\cite[Theorem 1.6.6, p.63]{fukushima2011}.
	\begin{lemma}\label{crit_cons}The markovian semigroup $(T_t)_{t>0}$ is conservative iff there exists a sequence $(f_n)_{n\in\mathbb{N}}$ 
		in $D(\mathcal{E})$ such that
		\begin{enumerate}[label=(\roman*)]
			\item $0\le f_n\le 1$ and $f_n\to \mathbbm{1}$ $dm$-almost everywhere as well as
			\item $\mathcal{E}(f_n,\psi)\to 0$ for all $\psi\in D(\mathcal{E})\cap L^1 (X,dm)$.
		\end{enumerate}
	\end{lemma}
	In the case $\alpha, \beta>-1$ we have $\mathbbm{1}\in D(\EE)$, such that the above criterion is clearly satisfied and therefore $(T_t)_{t>0}$ is conservative in this case. To treat the case $\alpha\le -1$ or $\beta \le -1$ we provide a suitable test function $\psi$ in the following lemma.
	\begin{lemma}\label{lemma_psi}
		The following holds.
		\begin{enumerate}[label=(\roman*)]
			\item Let $\alpha\le -1$ and $\beta >-1$. Then the function
			\begin{equation}\label{eq626} X\to \mathbb{R}, \, x\mapsto \frac{1}{\alpha}(d-x)^{-\alpha}
			\end{equation}
			is contained in $ D(\mathcal{E})\cap L^1(X,dm)$.
			\item Let $\alpha> -1$ and $\beta \le -1$. Then the function
			\begin{equation}\label{eq622}X\to \mathbb{R}, \, x\mapsto \frac{1}{\beta}x^{-\beta}
			\end{equation}
			is contained in $ D(\mathcal{E})\cap L^1(X,dm)$.
			\item Let $\alpha\le  -1$ and $\beta \le -1$. Then for all $y\in (0,d)$ the function
			\begin{equation}\label{eq621} X\to \mathbb{R}, \, x\mapsto \,
				y^{\beta}x^{-\beta}\mathbbm{1}_{\left(0,y\right]}(x)\,+\,
				\left(d-y\right)^{\alpha}(d-x)^{-\alpha}
				\mathbbm{1}_{\left(y,d\right)}(x)
			\end{equation}
			is contained in $ D(\mathcal{E})\cap L^1(X,dm)$.
		\end{enumerate}
	\end{lemma}
	\begin{proof}
		Let $\alpha\le -1 $, $\beta>-1 $ and $\psi$ be the function \eqref{eq626}. Then we  have
		\[
		\int_X\psi^2(x)\,dm(x)\,=\,
		\int_0^d \frac{x^\beta(d-x)^{-\alpha}}{\alpha^2d^{\alpha+\beta+1}}\, dx
		\,<\, \infty.
		\]
		Similarly,
		\[
		\int_X |\psi(x)|\,dm(x)\,=\,
		\int_0^d \frac{x^\beta}{-\alpha d^{\alpha+\beta+1}} \,dx
		\,< \, \infty,
		\]
		such that $\psi\in L^2(X,dm)\cap L^1(X,dm)$. The derivative of $\psi$ is given by $\psi'(x)=(d-x)^{-(\alpha+1)}$. We conclude by
		\[
		\int_X (\psi'(x))^2 \,dcm(x)\,
		=\,
		\int_{0}^d\,\frac{\sigma^2x^{\beta+1}(d-x)^{-(\alpha+1)}}{2d^{\alpha+\beta+1}}
		\,dx\,<\, \infty\]
		that $\psi\in D(\mathcal{E})$, which completes the proof of (i). Part (ii) can be verified analogously.
		We assume lastly that $\alpha,\beta\le -1$ and denote by $\psi_y$ the function \eqref{eq621}. We calculate
		\begin{gather*}
			\int_X \psi_y^2(x) \,dm(x) \,=\,
			\int_0^y\, \frac{ y^{2\beta}x^{-\beta}(d-x)^\alpha }{d^{\alpha+\beta+1}}\, dx\,+\,
			\int_y^d\, \frac{(d-y)^{2\alpha}x^\beta(d-x)^{-\alpha}}{ d^{\alpha+\beta+1}}\,dx\,<\, \infty.
		\end{gather*}
		Similarly, we obtain
		\begin{gather*}
			\int_X |\psi_y(x)|\,dm(x)\, =\,
			\int_0^y \,\frac{y^\beta(d-x)^\alpha }{d^{\alpha+\beta+1}} \,dx\,+\,
			\int_y^d\, \frac{(d-y)^\alpha x^\beta}{d^{\alpha+\beta+1}}\,dx\,< \,\infty
		\end{gather*}
		and hence $\psi_y\in L^2(X,dm)\cap L^1(X,dm)$. Note  that the weak derivative of $\psi_y$ is given by 
		\begin{equation}\label{eq43}
			\psi_y'(x)\,=\,-\beta y^{\beta} x^{-(\beta+1)}
			\mathbbm{1}_{\left(0,y\right]}(x)
			\,+\,\alpha 
			\left(d-y\right)^{\alpha}(d-x)^{-(\alpha+1)}
			\mathbbm{1}_{\left(y,d\right)}(x).
		\end{equation}
		The computation
		\begin{align*}&
			\int_X \,(\psi_y'(x))^2 \,dcm(x)\\ =\,&
			\int_0^y \, \frac{\sigma^2\beta^2y^{2\beta} x^{-(\beta+1)}(d-x)^{\alpha+1} }{2d^{\alpha+\beta+1}} \,dx\,+\,
			\int_y^d\, \frac{\sigma^2\alpha^2(d-y)^{2\alpha}x^{\beta+1}(d-x)^{-(\alpha+1)}}{2d^{\alpha+\beta+1}}\,dx\,<\, \infty
		\end{align*}
		yields $\psi_y\in D(\mathcal{E})$ and finishes the proof.
	\end{proof}
	
	\begin{thm}\label{not_cons}
		The semigroup $(T_t)_{t>0}$ is conservative if and only if $\alpha,\beta>-1$.
	\end{thm}
	\begin{proof}By our previous considerations the 'if'-part is clear and we are left to show the 'only if'-part. We assume that $\alpha\le -1$ or $\beta \le -1$ as well as that $(T_t)_{t>0}$ is conservative and lead this to a contradiction. By Lemma \ref{crit_cons} there exists a sequence of functions $(f_n)_{n\in\mathbb{N}}\subset D(\mathcal{E})$ such that $0\le f_n\le 1$, $f_n\to \mathbbm{1}$ $dm$-almost everywhere and $\mathcal{E}(f_n,\psi)\to 0$  for all $\psi\in D(\mathcal{E})\cap L^1(X, dm)$.
		
		We consider the case $\alpha\le -1$ and $\beta> -1$ and choose $\psi$ to be the function \eqref{eq626}. By inserting the derivative of $\psi$ it follows that
		\begin{gather}\label{eq44}0\,
			\leftarrow \,\int_X f_n'(x)\psi'(x) \,dcm(x)\,=\,\frac{\sigma^2}{2d^{\alpha+\beta+1}}\,
			\int_0^d\, f_n'(x)x^{\beta+1}\,dx
		\end{gather}
		as $n\to \infty$. For every $n$, the function $g_n(x)= \tilde{f}_n(x) x^{\beta+1}$   is weakly differentiable with weak derivative 
		\begin{equation}\label{eq99}
			g_n'(x)\,=\,
			f_n'(x)x^{\beta+1}\,+\,(\beta+1)f_n(x)x^\beta.
		\end{equation}
		The fundamental theorem of calculus as in \cite[Theorem 8.2, p.204]{brezis2010functional} holds for $g_n$, if we can verify that $g_n\in  L^1((0,d),dx)$ and $g_n' \in L^1((0,d),dx)$. The former follows immediatly from $\beta>-1$ and analogously  we obtain that $f_n(x)x^\beta$ is  in  $L^1((0,d),dx)$.  By Hölder's inequality we conclude that
		\[
		\int_0^d |f_n'(x)|x^{\beta+1}\, dx\,\le \,\left(\int_0^d (f_n'(x))^2 x^{\beta+1}(d-x)^{\alpha+1}\, dx \right)^\frac{1}{2} 
		\left(\int_0^d  x^{\beta+1}(d-x)^{-(\alpha+1)}\, dx \right)^\frac{1}{2}.
		\]
		The expression on the right-hand side is finite since  $f_n\in D(\mathcal{E})$ and hence  $g_n'\in L^1((0,d),dx)$. We conclude that
		\[
		\int_\epsilon^{d-\epsilon} f_n'(x)x^{\beta+1}+(\beta+1)f_n(x)x^\beta\, dx\,=\, \tilde{f}_n(d-\epsilon)(d-\epsilon)^{\beta+1}\,-\,\tilde{f}_n(\epsilon)\epsilon^{\beta+1}
		\]
		for $\epsilon >0$. In light of Corollary \ref{limiz} and our assumptions on $\alpha, \beta$ the right hand-side of the above inequality tends to $0$ ad $\epsilon\searrow 0$. It follows
		\[
		\int_0^d f_n'(x)x^{\beta+1}\,+\,(\beta+1)f_n(x)x^{\beta+1}\, dx\,=\,0
		\]
		for each $n$ and by \eqref{eq44} consequently
		\[
		\lim_{n\to\infty}\,\int_0^df_n(x)x^{\beta+1}\,dx\,=\,0.
		\]
		But this is a contradiction to
		 $f_n\to \mathbbm{1}$ $dm$-almost everywhere. The case $\alpha >-1$ and $\beta\le -1$ can be treated analogously by choosing $\psi$ as  \eqref{eq622}.
		
		Next, we consider the case $\alpha,\beta\le -1$, take $y\in (0,d)$ such that $\tilde{f}_n(y)\to 1$  and choose $\psi_y$ as the corresponding function \eqref{eq621}. The weak derivative of $\psi_y$ is then given by \eqref{eq43} and hence we can conclude that
		\begin{align}\begin{split}
				\label{eq45}
				0\,\leftarrow\,
				&\int_X f_n'(x)\psi_y'(x) \,dcm(x)\\=\,-&\int_{0}^y\,\frac{\sigma^2\beta y^\beta f_n'(x) (d-x)^{\alpha+1}}{2 d^{\alpha+\beta+1}}\,dx\,+\,\int_{y}^{d}\, \frac{\sigma^2\alpha (d-y)^{\alpha}f_n'(x)x^{\beta+1}
				}{2d^{\alpha+\beta+1}}\,dx
			\end{split}
		\end{align}
		as $n\to \infty$. For $n\in\mathbb{N}$ we let again
		$g_n(x)= \tilde{f}_n(x)x^{\beta+1}$ and recall that $g_n$ is weakly differentiable  with weak derivative \eqref{eq99}. As before we  conclude that $g_n$, $f_n(x)x^\beta$ and $f_n'(x)x^{\beta+1}$ are in $L^1((y,d),dx)$ and obtain
		\[
		\int_{y+\epsilon}^{d-\epsilon} f_n'(x)x^{\beta+1}+(\beta+1)f_n(x)x^{\beta}\, dx\,=\, \tilde{f}_n(d-\epsilon)(d-\epsilon)^{\beta+1}\,-\,\tilde{f}_n(y+\epsilon)(y+\epsilon)^{\beta+1}
		\] 
		for  $\epsilon>0$. Due to Corollary \ref{limiz}, the term $\tilde{f}_n(d-\epsilon)(d-\epsilon)^{\beta+1}$ tends to $0$ as $\epsilon\searrow 0$ and hence we have
		\[
		\int_{y}^{d} f_n'(x)x^{\beta+1}\, dx\,=\,-\int_{y}^{d}(\beta+1)f_n(x)x^{\beta} \,dx\,-\,\tilde{f}_n(y)y^{\beta+1}.
		\]
		By inserting this in the latter  term from the right-hand side of  \eqref{eq45} we obtain that
		\begin{align}\label{eq0}
			\int_{y}^{d} \frac{\sigma^2\alpha (d-y)^{\alpha}f_n'(x)x^{\beta+1}
			}{2d^{\alpha+\beta+1}}\,dx\,=\,
			\frac{-\sigma^2\alpha (d-y)^{\alpha}
			}{2d^{\alpha+\beta+1}}\,\left[
			\int_{y}^{d}(\beta+1)f_n(x)x^{\beta} \,dx\,+\,\tilde{f}_n(y)y^{\beta+1}\right].
		\end{align}
	Our choice of $y$ implies that
		\[
		\int_{y}^{d}(\beta+1)f_n(x)x^{\beta} \,dx\,+\,\tilde{f}_n(y)y^{\beta+1}\,\to\,
		\int_{y}^{d}(\beta+1)x^{\beta} \,dx\,+\, y^{\beta+1}\,=\, d^{\beta+1}
		\]
		as $n\to \infty$ and therefore 
		\[
		\lim_{n\to\infty}\,
		\int_{y}^{d} \,\frac{\sigma^2\alpha (d-y)^{\alpha}f_n'(x)x^{\beta+1}
		}{2d^{\alpha+\beta+1}}\,dx\,=\,
		\frac{-\sigma^2\alpha (d-y)^{\alpha}
		}{2d^{\alpha}}>0.\]
		We  verify that also the first term from the right-hand side of \eqref{eq45} converges to a positive number, which will indeed  be a contradiction.
	For this purpose, we consider the function $h_n(x)= \tilde{f}_n(x)(d-x)^{\alpha+1}$, which has weak derivative
		\[h_n'(x)\,=\,
		f_n'(x)(d-x)^{\alpha+1}-(\alpha+1)f_n(x) (d-x)^{\alpha}.
		\]
		Analogously as for $g_n$, we can conclude that
		$h_n$, $f_n(x) (d-x)^{\alpha}$ and $f_n'(x)(d-x)^{\alpha+1}$  are elements of $L^1((0,y),dx)$ and obtain therefore 
		\begin{align*}&
			\int_{\epsilon}^{y-\epsilon}f_n'(x)(d-x)^{\alpha+1}-(\alpha+1)f_n(x) (d-x)^{\alpha}\,dx\\=\,& \tilde{f}_n(y-\epsilon)(d-y+\epsilon)^{\alpha+1}-\tilde{f}_n(\epsilon)(d-\epsilon)^{\alpha+1}
		\end{align*}
		for $\epsilon >0$.
		Corollary \ref{limiz} yields that $\tilde{f}_n(\epsilon)(d-\epsilon)^{\alpha+1}$ converges to $0$ as $\epsilon \searrow 0$ and consequently we have
		\[
		\int_0^y f_n'(x)(d-x)^{\alpha+1}\, dx\,=\,\int_{0}^{y}(\alpha+1)f_n(x)(d-x)^{\alpha} \,dx\,+\,\tilde{f}_n(y)(d-y)^{\alpha+1}.
		\]
		It follows  that
		\[
		\int_{0}^y\frac{\sigma^2\beta y^\beta f_n'(x) (d-x)^{\alpha+1}}{2 d^{\alpha+\beta+1}}\,dx\,=\,
		\frac{\sigma^2\beta y^\beta}{2 d^{\alpha+\beta+1}}
		\left[
		\int_{0}^{y}(\alpha+1)f_n(x)(d-x)^{\alpha} \,dx+\tilde{f}_n(y)(d-y)^{\alpha+1}
		\right].
		\]
		Due to our choice of $y$ we have that
		\[
		\int_{0}^{y}(\alpha+1)f_n(x)(d-x)^{\alpha} \,dx+\tilde{f}_n(y)(d-y)^{\alpha+1}\,\to\,
		\int_{0}^{y}(\alpha+1)(d-x)^{\alpha} \,dx+ (d-y)^{\alpha+1}\,=\, d^{\alpha+1}
		\]
		as $n\to \infty$ and therefore
		\[\lim_{n\to\infty}\,
		-\int_{0}^y\,\frac{\sigma^2\beta y^\beta f_n'(x) (d-x)^{\alpha+1}}{2 d^{\alpha+\beta+1}}\,dx\,=\, \frac{-\sigma^2\beta y^\beta}{2 d^{\beta}}\,>\,0.
		\]
	\end{proof}
	
	We continue by showing irreducibility of $(T_t)_{t>0}$, i.e. that $(T_t)_{t>0}$ has only trivial invariant sets, for  details see \cite[pp.53,55]{fukushima2011}.
	For this we employ the following criterion from \cite[Theorem 1.6.1, p.54]{fukushima2011}.
	\begin{lemma}
		A $dm$-measurable set $B\subset X$ is $(T_t)_{t>0}$-invariant if and only if 
		$\mathbbm{1}_B f\in D(\EE)$  and
		\[
		\EE(f,f)\,=\, \EE(\mathbbm{1}_B f,\mathbbm{1}_B f)\,+\,\EE(\mathbbm{1}_{X\setminus B} f,\mathbbm{1}_{X\setminus B} f)
		\]
		for every $f\in D(\EE)$.
	\end{lemma}
	This implies together with Proposition \ref{cont2} that if $B$ is a $(T_t)_{t>0}$-invariant, $dm$-measurable subset of $X$, the function $\mathbbm{1}_Bf$ has a $dm$-version which is  continuous on $\tilde{X}$. Since we can choose $f$ to be $1$ on $(\epsilon, d-\epsilon)$ it follows that either $dm(B\cap (\epsilon, d-\epsilon))=0$ or $dm((\epsilon, d-\epsilon)\setminus B)=0$ for any $\epsilon>0$. So either $B$ is a null-set or $B$ has full measure, i.e. the following holds.
	\begin{prop}\label{irred}
		The semigroup $(T_t)_{t>0}$ is irreducible. 
	\end{prop}
	As a consequence $(T_t)_{t>0}$ is either transient or recurrent, see \cite[Lemma 1.6.4 (iii), p.55]{fukushima2011}. To determine which is the case we make use of the following criteria, see \cite[Theorem 1.6.2 and 1.6.3, p.58]{fukushima2011}.
	To state these criteria we denote the extended Dirichlet space of $(\EE,D(\EE))$ by $(D_e(\EE), \EE)$, for a definition see \cite[p.41]{fukushima2011}.
	\begin{lemma}\label{rec_and_trans} The  semigroup $(T_t)_{t>0}$ is transient iff any $f\in D_e(\mathcal{E})$ with $\mathcal{E}(f,f)=0$ satisfies $f=0$ $dm$-almost everywhere. The semigroup $(T_t)_{t>0}$ is recurrent iff $\mathbbm{1}\in D_e(\mathcal{E})$ and $\mathcal{E}(\mathbbm{1},\mathbbm{1})=0$.
	\end{lemma}
	Clearly, in the situation that $\alpha, \beta>-1$ we have $\mathbbm{1}\in D(\EE)$ with $\EE(\mathbbm{1},\mathbbm{1})=0$ such that $(T_t)_{t>0}$ is recurrent. The next proposition treats also the other cases.
	\begin{prop}\label{trans_rec} The semigroup $(T_t)_{t>0}$ is recurrent iff $\alpha,\beta>-1$ and transient iff $\alpha\le -1$ or $\beta\le -1$.
	\end{prop}
	\begin{proof}
		By our previous considerations it suffices to show that $(T_t)_{t>0}$ is transient, if $\alpha\le -1$ or $\beta\le -1$. 
		We assume $\alpha\le -1$ and let $f\in D_e(\mathcal{E})$ with $\mathcal{E}(f,f)=0$. If we can verify that $f=0$ $dm$-almost everywhere, Lemma \ref{rec_and_trans} yields that $(T_t)_{t>0}$ is transient. By definition of the extended Dirichlet space there exists an $\mathcal{E}$-Cauchy sequence $(f_n)_{n\in\mathbb{N}}\subset D(\EE)$ with $f_n\to f$ $dm$-almost everywhere. Since  $\alpha\le -1$ it follows that $(f_n')_{n\in\mathbb{N}}\to 0$ in $L^2_{\loc}((0,d], dx)$. By Corollary \ref{limiz} we have $\lim_{x\nearrow d} \tilde{f}_n(x)=0$ for each $n$ and  therefore  Poincare's inequality as in \cite[Theorem 8.16, p.242]{altEN} applies to each of the $f_n$. Consequently, $f_n \to 0$ in $H^1_{\loc}((0,d])$ and hence indeed $f=0$. The case $\beta\le -1$ can be treated analogously.
	\end{proof}
\subsection{Spectral gap and Asymptotics}
If $\alpha, \beta>-1$,  the corresponding (rescaled) family of Jacobi polynomials is given by
\[
Q_n(x)\,=\, {}_2F_1\left(-n, n+\alpha+\beta+1; \alpha+1;1-\frac{x}{d}\right) ,\;\; n\in \mathbb{N}_0.
\]
 This is a complete orthogonal basis of $L^2(X, dm)$ and moreover we have 
\[
GQ_n(x)\,=\,\frac{-\sigma^2 n(n+\alpha+\beta+1)}{2} Q_n(x)
\]
for $x\in (0,d)$,
see \cite[Theorem 6.4.3, p.299; Theorem 6.5.2, p.307]{andrews1999special} and \cite[Eq. (6.3.8), p.297]{andrews1999special}. The following lemma shows, that $(Q_n)_{n\in \mathbb{N}}$ is  even an orthogonal system of eigenfunctions of $(L, D(L))$ in this case.
\begin{lemma}\label{gen_Lemma}
	The following holds.
	\begin{enumerate}[label=(\roman*)]\item If $\alpha, \beta >-1$,  it holds $(G,C^\infty([0,d]))\subset (L,D(L))$.
		\item If $\alpha >-1$,  it holds $(G,C_c^\infty((0,d]))\subset (L,D(L))$.
		\item If $\beta >-1$,  it holds $(G,C_c^\infty([0,d)))\subset (L,D(L))$.
	\end{enumerate}
\end{lemma}
\begin{proof} We provide a proof of (i), the remaining parts can be shown analogously.
	We assume $\alpha, \beta >-1$ and let $f\in C^\infty([0,d])$. Then we have $Gf\in C^0([0,d])$ and in particular $f, Gf\in L^2(X,dm)$ since $dm$  has finite total mass. Similarly, we conclude $f\in D(\EE)$ and the claim follows if we can verify that 
	\begin{equation}\label{eq949}
		\mathcal{E}(f,g)\,=\,-\left(Gf,g\right)_{L^2(X,dm)}
	\end{equation}
	for any $g\in D(\EE)$ by  \cite[Proposition 10.4 (ii), p.225]{schmüdgen2012unbounded} . Due to \eqref{eq37} we have
	\[
	(f'cm)'(x)\,=\,\big[c(x)f''(x)+(a-bx)f'(x)\big]m(x)
	\]
	for all $x\in (0,d)$. The integration by parts formula  \eqref{int_by_parts} yields that
	\begin{equation}\label{eq747}\int_\epsilon^{d-\epsilon}  g'(x)f'(x)\, dcm(x)\,=\, -
		\int_\epsilon^{d-\epsilon} g(x)Gf (x)\, d m(x) \,+\,
		\big[
		\tilde{g}f'cm
		\big]_\epsilon^{d-\epsilon}
	\end{equation}
	for  $\epsilon >0$. Observe that $f'(x)$ converges as  $x\searrow 0$ and $x\nearrow d$. Furthermore, we get $\lim_{x\searrow 0}\tilde{g}cm(x)=\lim_{x\nearrow d}\tilde{g}cm(x)=0$ by Corollary \ref{limiz}. Hence, taking $\epsilon\searrow 0$ in \eqref{eq747} yields \eqref{eq949}.
\end{proof}
Using spectral theory, we obtain the explicit representation of the semigroup 
\begin{equation}\label{Eq_3}
T_t f\,=\, \sum_{n=0}^{\infty} \frac{e^{\frac{-\sigma^2 n(n+\alpha+\beta+1) t}{2}} \left(f, Q_n \right)_{L^2(X, dm)}}{\|Q_n\|_{L^2(X, dm)}^2} Q_n
\end{equation}
for $f\in L^2(X, dm)$ by \cite[Proposition 5.12, p.94]{schmüdgen2012unbounded} whenever $\alpha, \beta>-1$. 

	\begin{cor}The following holds for every   $t> 0$ and $f\in L^2(X, dm)$.
		\begin{enumerate}[label=(\roman*)]
			\item If $\alpha, \beta >-1$ we have \begin{equation}\label{Eq_5}
			\left\|T_t f -\frac{\int_X f\, dm}{dm(X)}\mathbbm{1}\right\|_{L^2(X, dm)}\, \le \, e^{-bt} \|f\|_{L^2(X, dm)}.
			\end{equation}
			\item If $\alpha\le -1$, $\beta>-1$ or $\alpha>-1$, $ \beta \le -1$ we have \[
			\left\|T_t f \right\|_{L^2(X, dm)}\, \le \, e^{\frac{-t}{C(\alpha, \beta, \sigma)}} \|f\|_{L^2(X, dm)},
			\]
			where $C(\alpha, \beta,\sigma)$ is the constant from the right-hand side of \eqref{Eq_10}.
		\end{enumerate}
	\end{cor}
\begin{proof}
	For (i) we observe that  $Q_0= \mathbbm{1}$ and consequently \eqref{Eq_3} yields
	\begin{align*}		
		\left\|T_t f -\frac{\int_X f\, dm}{dm(X)}\mathbbm{1}\right\|_{L^2(X, dm)}^2\,=\, 
		 \sum_{n=1}^{\infty} e^{-\sigma^2 n(n+\alpha+\beta+1) t} \frac{\left(f, Q_n \right)_{L^2(X, dm)}^2}{\|Q_n\|_{L^2(X, dm)}^2}\, \le \,
		 e^{-2bt} \|f\|^2_{L^2(X,dm)}.
	\end{align*}
In the last inequality we used additionally that
${\sigma^2(\alpha+\beta+2)}=2b$ by \eqref{EqAlphaBeta}. We conclude that \eqref{Eq_5} holds by taking the square-root. Part (ii) is an immediate consequence from Lemma \ref{Lemma_Poinc} together with \cite[Theorem 1.1.1, p.24]{Wang2005}.
\end{proof}
	\section{The corresponding process}\label{Sec_process}
	In this section we analyze a  $dm$-symmetric Hunt process, which is associated to the Dirichlet form $(\mathcal{E},D(\mathcal{E}))$ in the sense that its transition semigroup determines $(T_t)_{t>0}$ as in \cite[Lemma 1.4.3, p.30]{fukushima2011}. To this end, we adjoin the cemetery  $\Delta$ as Alexandroff point to the state space $X$ and write $X_\Delta=X\cup\{ \Delta\}$. Moreover, we let $\mathbf{M}=\left(\Omega,\mathfrak{A},(Y_t)_{t\in[0,\infty]},(P_x)_{x\in X_\Delta}\right)$ be a Hunt process associated to $(\EE,D(\EE))$, for details on Hunt processes see \cite[Appendix A.2, pp.384-405]{fukushima2011}. We denote its transition function on $X$ by $(\rho_t)_{t>0}$ and its life time by $\zeta$.
	\begin{rem}
		The Dirichlet form $(\EE, D(\EE))$ is regular by Proposition \ref{is_reg} and therefore there exists an associated Hunt  process. A construction of it can be found in \cite[Chapter 7, pp.369- 381]{fukushima2011}.
	\end{rem}
\subsection{Basic properties}
We use the convention that $f(\Delta)=0$ for any function $f$, which is a priori defined on $X$.  Moreover for the notion of a set being (properly) exceptional with respect to $\mathbf{M}$ we refer to \cite[pp.152-153]{fukushima2011}. We note that every properly exceptional set is exceptional and every exceptional set is contained in a properly exceptional set, for the latter see \cite[Theorem 4.1.1, p.155]{fukushima2011}. Moreover we write $Y_{\zeta-}$ for the left limit of the process $Y$ at  $\zeta$.
\begin{thm}\label{THM_proc1} The following holds.
	\begin{enumerate}[label=(\roman*)]
		\item
		A set $B\subset X$ is exceptional with respect to $\mathbf{M}$ iff $B\subset X\setminus \tilde{X}$.
		\item There exists a properly exceptional set $N\subset X$, such that $\rho_t(x,\cdot)$ is absolutely continuous with respect to $dm$ for every $t>0$ and $x\in X\setminus N$.
		\item If $\alpha, \beta >-1$,	we have  $P_x(\{\zeta <\infty\})=0$ for quasi every $x\in X$.
		\item The path $[0, \zeta)\to X, t\mapsto Y_t$ is $P_x$-almost surely continuous for quasi every $x\in X$.
		\item It holds $P_x\left(\left\{Y_{\zeta-} \in X \wedge \zeta  <\infty\right\}\right)=0$ for quasi every $x\in X$.
		\item 
		If $\alpha, \beta >-1$ and $N$ as in (ii), then \[
		\int_0^\infty 1_B(Y_s)\, ds \, =\,\infty
		\]$P_x$-almost surely for every $x\in X\setminus N$ and $B\in \mathfrak{B}(X)$ with $dm(B)>0$.
			\item If $\alpha\le -1\vee  \beta\le -1 $ and $r,s\in \RR$ according to Lemma \ref{lemma_ref_fct}, then
		\begin{equation}\label{Eq10}
		\int_0^\infty Y_t^r(d-Y_t)^s\, dt<\infty
		\end{equation}
		$P_x$-almost surely for quasi every $x\in X$.
		\item If $\alpha, \beta >-1$,  then 
		\[
		\lim_{t\to \infty}\frac{1}{t}\int_0^t f(Y_s)\, ds\,=\, \frac{1}{dm(X)}\int_X f(x)\, dm(x)
		\]
		$P_x$-almost surely for quasi every $x\in X$ for every $\mathfrak{B}$-measurable $f\in L^1(X, dm)$.
	\end{enumerate}
\end{thm}
\begin{rem}
	We note that due to the convention $f(\Delta)=0$ the integrand in \eqref{Eq10} vanishes as soon as $Y_t=\Delta$.
\end{rem}
\begin{proof}
	Part (i) is a consequence of Theorem \ref{quasi_notions} (i) together with \cite[Theorem 4.2.1 (ii), p.161]{fukushima2011}. Part (ii) follows from Theorem \ref{sob_emb} and \cite[Theorem 4.2.7, p.166]{fukushima2011}. 
	By Theorem \ref{not_cons} and \cite[Exercise 4.5.1, p.187]{fukushima2011} we obtain (iii).
	Since $(\EE, D(\EE))$ is strongly local as remarked in the introductory section, part (iv) and (v) follow from \cite[Theorem 4.5.3, p.186]{fukushima2011}. Part (vi) follows from \cite[Lemma 4.8.1, p.209]{fukushima2011} together with the Propositions \ref{irred} and \ref{trans_rec}.
	For (vii) we choose strictly positive functions $(\varphi_n)_{n\in\mathbb{N}}$ on $X$ with $\varphi_n(x)\nearrow 1$ for every $x\in X$ such that 
	$
	f_n(x)\,=\,\varphi_n(x)f(x) $
	is bounded on $X$ and integrable with respect to $dm$ for every $n\in \mathbb{N}$, where $ f(x)\,=\,x^r(d-x)^s
$. If we choose $	C_{\alpha, \beta, \sigma, d, r,s}$ as in \eqref{Eq_Hardy} the function 
	\[
	\frac{f_n}{	C_{\alpha, \beta, \sigma, d, r,s}}
	\]
	is a reference function of $(\EE,D(\EE))$ for every $n\in \mathbb{N}$, see \cite[p.40]{fukushima2011} for a definition. By \cite[Theorem 1.5.1, p.40]{fukushima2011} it holds
	\[
	\int_X
	{f_n} \left(\lim_{\lambda\searrow 0 }  R_\lambda
	{f_n} \right) \, dm\, \le \, C_{\alpha, \beta, \sigma, d, r,s}^2,
	\]
	where $(R_\lambda)_{\lambda>0}$ denotes the resolvent associated to $(T_t)_{t>0}$, for details see \cite[Section 1.3, pp.16-25]{fukushima2011}, and the limit is attained almost everywhere. We have also
	\begin{equation*}
		R_\lambda f_n(x) \,=\, E_x\left(\int_0^\infty e^{-\lambda t}f_n(Y_s)\,ds\right)\, \to \, E_x\left(\int_0^\infty f_n(Y_s)\,ds\right)
	\end{equation*}
	 as $\lambda\searrow 0$ for almost every $x\in X$ by \cite[Theorem 4.2.3, p.162]{fukushima2011} and monotone convergence. 
	 Employing again the monotone convergence theorem we conclude that
	 \[C_{\alpha, \beta, \sigma, d, r,s}^2\, \ge \,
	 	\int_X
	 {f_n}(x) E_x\left(\int_0^\infty f_n(Y_s)\,ds\right) \, dm(x)\, \to  \, 
	 	\int_X
	 {f}(x) E_x\left(\int_0^\infty f(Y_s)\,ds\right) \, dm(x).
	 \]
	 Therefore, Proposition \ref{trans_rec} and
	 \cite[Theorem 4.2.6, p.164]{fukushima2011} yield that
	 \[
	 E_x\left(\int_0^\infty f(Y_s)\,ds\right)
	 \]
	is quasi-continuous in $x$ and in particular finite for quasi every $x\in X$. We conclude that the integrand 
	\[
	\int_0^\infty f(Y_s)\,ds
	\]
	is finite $P_x$-almost surely.
	Part (viii) is a consequence of Propositions \ref{irred} and \ref{trans_rec} together with \cite[Theorem 4.7.3, (iii), p.205]{fukushima2011}.
\end{proof}

For a nearly Borel set $B\subset X$, see \cite[p.392]{fukushima2011} for the definition, we define the $\lambda$-order hitting distribution
\begin{equation}\label{eq635}
H^\lambda_B (x,E) \,=\,E_x(e^{-\lambda \tau_B} \mathbbm{1}_E(X_{\tau_B})) 
\end{equation}
for positive $\lambda >0$ and universally measurable subsets $E$ of $X$. The appearing random time $\tau_B$ is defined by
\[
\tau_B \,=\, \inf\{t>0 | X_t\in B\}
\]
and is a stopping time with respect to the minimum completed admissible filtration of $\mathbf{M}$, see \cite[Theorem A.2.3, p.391]{fukushima2011}.
Moreover, we let 
$\mathcal{H}^\lambda_B$ be the orthogonal complement of
\begin{equation*}
	\left\{
	f\in D(\mathcal{E})\big|\,\tilde{f}=0\text{  q.e. on }B
	\right\}
\end{equation*}
in $(D(\mathcal{E}),\mathcal{E}_\lambda)$. 
In the next lemma, we characterize the spaces  $\mathcal{H}^\lambda_{\{0\}}$ and $\mathcal{H}^\lambda_{\{d\}}$. The proof relies on the same technique employed in the proof of Theorem \ref{ortho}.

\begin{lemma}\label{o2}Let $\lambda>0$.
	\begin{enumerate}[label=(\roman*)]
		\item
		If $\alpha>-1$, we have \[\mathcal{H}^\lambda_{\{d\}}=\begin{cases}\spa \{\xi_\lambda\}, & \alpha < 0, \\\{0\},& \alpha \ge 0.
		\end{cases}\]
		\item If $\beta>-1$, we have \[\mathcal{H}^\lambda_{\{0\}}=\begin{cases}\spa \{\eta_\lambda\}, & \beta <0, \\\{0\},& \beta \ge 0.
		\end{cases}\]
	\end{enumerate}
\end{lemma}
\begin{proof}
	We prove (i), the proof of (ii) works analogously. Assume first $\alpha \ge  0$. In this case the set $\{d\}$ has capacity $0$ by Theorem \ref{quasi_notions} (i). Therefore the space \begin{equation}\label{eq444}\big\{
		f\in D(\mathcal{E})\big|\,\tilde{f}=0 \text{ q.e. on } \{d\}\}
	\end{equation} is whole $D(\EE)$  and consequently $\mathcal{H}^\lambda_{\{d\}}=\{0\}$. For the remaining case we assume that  $-1<\alpha<0$. If $\beta\le -1$ or $\beta \ge 0$,  the spaces \eqref{eq444}
	 and
	$\mathcal{F}$ coincide
	due to Theorem \ref{DE}. The claim follows by employing Theorem \ref{ortho}. 
	Consider lastly the case that additionally $-1<\beta<0$. Then \eqref{eq444} has codimension $1$ since it is the nullspace of the point evaluation at $d$. Hence the claim follows when we can show that $\spa\{\xi_\lambda\}$ and \eqref{eq444} are orthogonal to each other in $(D(\mathcal{E}),\mathcal{E}_\lambda)$. Let $f$ be an element of \eqref{eq444}, then  \eqref{int_by_parts} and \eqref{eq2} yield that
	\[\big[\tilde{f}\xi_\lambda' cm \big]^{d-\epsilon}_\epsilon \,=\,
	\int_\epsilon^{d-\epsilon}
	f'\xi_\lambda'cm(x)+\lambda f\xi_\lambda  m(x) \,dx \,\to\,
	\mathcal{E}_\lambda(f,\xi_\lambda)
	\]
	as $\epsilon \searrow 0$. By Lemma \ref{prop_f1} \ref{r_boundary2} together with $\tilde{f}(d)=0$ we get $\lim_{x\nearrow d} \tilde{f}\xi_\lambda' cm(x)=0$. Lemma \ref{prop_f1} \ref{l_boundary2} together with the continuity of $\tilde{f}$ at $0$ implies that $\lim_{x\searrow 0} \tilde{f}\xi_\lambda' cm(x)=0$. Hence  $\mathcal{E}_\lambda(f,\xi_\lambda)=0$, which completes the proof of (i). 
\end{proof}
Using the relation between orthogonal projections on the spaces from the previous Lemma and the $\lambda$-order hitting distribution of the corresponding set given by \cite[Theorem 4.3.1, p.168]{fukushima2011}, we calculate the hitting probabilities of the boundary points.

\begin{thm}\label{Prob_hitting}The following holds.
	\begin{enumerate}[label=(\roman*)]
		\item If  $-1<\alpha<0$,  we have for quasi-every $x\in X$ that
		\begin{align*}&
			P_x\left(\{\tau_{\{d\}}<\infty\}\right)\\=\,&\begin{cases}
				1,& \beta>-1,
				\\1,& \beta\le -1,\;x=d,\\
				\frac{\Gamma\left(-\alpha-\beta\right)}{\Gamma(1-\beta)\Gamma(-\alpha)}\,\left(\frac{x}{d}\right)^{-\beta}
				{}_2F_1\left(
				\alpha+1,-\beta;1-\beta;\frac{x}{d}
				\right),& \beta\le -1,\;x<d.
			\end{cases}
		\end{align*}
		\item If  $-1<\beta<0$,  we have for quasi-every $x\in X$ that
		\begin{align*}&
			P_x\left(\{\tau_{\{0\}}<\infty\}\right)\\=\,&\begin{cases}
				1,& \alpha>-1,
				\\1,& \alpha\le -1,\;x=0,\\
				\frac{\Gamma\left(-\alpha-\beta\right)}{\Gamma(1-\alpha)\Gamma(-\beta)}\,\left(1-\frac{x}{d}\right)^{-\alpha}
				{}_2F_1\left(
				\beta+1,-\alpha;1-\alpha;1-\frac{x}{d}
				\right),& \alpha\le -1,\;x>0.
			\end{cases}
		\end{align*}
	\end{enumerate}
\end{thm}
\begin{rem}
	The case $-1<\alpha<0$ is the only one, in which the hitting probability of $\{d\}$ can be non-trivial. Indeed the case $\alpha \le -1$ is trivial by choice the of $X$ and in the case $\alpha \ge 0$ the set $\{d\}$ is exceptional by Theorem \ref{THM_proc1} (i). Therefore, $P_x(\{\tau_{\{d\}}<\infty\})=0$ for quasi-every starting point $x$.
\end{rem}
\begin{proof}
	We assume that $-1<\alpha<0$ and let $f\in D(\mathcal{E})$ with $\tilde{f}(d)=1$. In the following, we will first identify $P_{\mathcal{H}^\lambda_{\{d\}}} f$, the orthogonal projection of $f$ onto $\mathcal{H}^\lambda_{\{d\}}$ in $(D(\EE), \EE_\lambda)$. Due to \cite[Theorem 4.3.1, p.168]{fukushima2011} the function
	\begin{equation}\label{eq123}
		H_{\{d\}}^\lambda \tilde{f}(x)=E_x(e^{-\lambda \tau_{\{d\}}} \tilde{f}(X_{\tau_{\{d\}}}))
	\end{equation} 
	exists for quasi-every $x\in X$ and defines a quasi-continuous version of  $P_{\mathcal{H}^\lambda_{\{d\}}}f$. In the following, we distinguish different cases of $\beta$ to identify \eqref{eq123} and  by  $\tilde{f}(X_{\tau_{\{d\}}})=1$ also the limit
	\begin{equation}\label{eq41}H_{\{d\}}^\lambda \tilde{f}(x)\,\to \, P_x(\{\tau_{\{d\}}<\infty\})
	\end{equation} as $\lambda \searrow 0$.

	Let $\beta>-1$ and take $\lambda>0$. The space $\mathcal{H}^\lambda_{\{d\}}$ coincides with $\spa\{\xi_\lambda\}$ by Lemma \ref{o2} (i). Since 
$
	f-P_{\mathcal{H}^\lambda_{\{d\}}}f $ is an element of \eqref{eq444} we have
	 $\tilde{f}(d)-\widetilde{P_{\mathcal{H}^\lambda_{\{d\}}} f}(d)=0$. 
	 Theorem \ref{prop_f1} \ref{r_boundary} yields 
	\[\lim_{x\nearrow d} \xi_\lambda (x)\,=\,
	\frac{\Gamma(\beta+1)\Gamma(-\alpha)}{\Gamma\left(\frac{-\alpha+\beta+1}{2}+\gamma\right)\Gamma\left(\frac{-\alpha+\beta+1}{2}-\gamma\right)}\,>\,0,
	\] such that necessarily
	\begin{equation*}
		P_{\mathcal{H}^\lambda_{\{d\}}}f\,=\,
		\frac{\Gamma\left(\frac{-\alpha+\beta+1}{2}+\gamma\right)\Gamma\left(\frac{-\alpha+\beta+1}{2}-\gamma\right)}{\Gamma(\beta+1)\Gamma(-\alpha)} \,\xi_\lambda.
	\end{equation*}
	The quasi-continuous version 
	$H_{\{d\}}^\lambda \tilde{f}$ of $P_{\mathcal{H}^\lambda_{\{d\}}}f$ has therefore to coincide with
	\begin{align*}&
		\begin{cases}\frac{\Gamma\left(\frac{-\alpha+\beta+1}{2}+\gamma\right)\Gamma\left(\frac{-\alpha+\beta+1}{2}-\gamma\right)}{\Gamma(\beta+1)\Gamma(-\alpha)}\,
			{}_2F_1\left(
			\frac{\alpha+\beta+1}{2}+\gamma,
			\frac{\alpha+\beta+1}{2}-\gamma;\beta+1;\frac{x}{d}
			\right),&x<d,
			\\1,& x=d
		\end{cases}
	\end{align*}
for quasi-every $x\in X$. If we let $\lambda\searrow0$ to calculate \eqref{eq41} the corresponding paramater
	\begin{equation}\label{eq491}
		\gamma=\sqrt{\left(\frac{\alpha+\beta+1}{2}\right)^2-\frac{2\lambda}{\sigma^2}}
	\end{equation}
	tends towards $ \left|\frac{\alpha+\beta+1}{2}\right|$. Analyticity of the $\Gamma$-function and $_2F_1$ in its first two parameters, see \cite[p.65]{andrews1999special}, imply that
	\[
	P_x(\{\tau_{\{d\}}<\infty\})\,=\,
	\begin{cases}
		{}_2F_1\left(
		\alpha+\beta+1,0;\beta+1;\frac{x}{d}
		\right),&x<d,
		\\1,& x=d
	\end{cases}
	\]for quasi-every $x\in X$. The claimed identity  follows since ${}_2F_1\left(
	\alpha+\beta+1,0;\beta+1;\frac{x}{d}
	\right)=1$ by definition.
	
	Next, we consider $\beta\le -1$. Analogously to the previous case we conclude  by 	 $\widetilde{P_{\mathcal{H}^\lambda_{\{d\}}} f}(d)=1$ and Theorem \ref{prop_f1} \ref{r_boundary} that
	\begin{equation*}
		P_{\mathcal{H}^\lambda_B}f\,=\,
		\frac{\Gamma\left(\frac{-\alpha-\beta+1}{2}+\gamma\right)\Gamma\left(\frac{-\alpha-\beta+1}{2}-\gamma\right)}{\Gamma(1-\beta)\Gamma(-\alpha)} \,\xi_\lambda.
	\end{equation*}
	It follows that $	H_{\{d\}}^\lambda \tilde{f}(x)$ equals
	\begin{align*}&
	\begin{cases}\frac{\Gamma\left(\frac{-\alpha-\beta+1}{2}+\gamma\right)\Gamma\left(\frac{-\alpha-\beta+1}{2}-\gamma\right)}{\Gamma(1-\beta)\Gamma(-\alpha)}\,
			\left(\frac{x}{d}\right)^{-\beta}{}_2F_1\left(
			\frac{\alpha-\beta+1}{2}+\gamma,
			\frac{\alpha-\beta+1}{2}-\gamma;1-\beta;\frac{x}{d}
			\right),&x<d,
			\\1,& x=d
		\end{cases}
	\end{align*}
	for quasi-every $x\in X$. Taking the limit $\lambda\searrow 0$  yields that
	\[
	P_x(\{\tau_{\{d\}}<\infty\})\,=\,
	\begin{cases}\frac{\Gamma\left(-\alpha-\beta\right)}{\Gamma(1-\beta)\Gamma(-\alpha)}\,\left(\frac{x}{d}\right)^{-\beta}
		{}_2F_1\left(
		\alpha+1,-\beta;1-\beta;\frac{x}{d}
		\right),&x<d,
		\\1,& x=d,
	\end{cases}
	\]
	again for quasi-every $x\in X$.
	This finishes the proof of  (i), part (ii) can be shown analogously.
\end{proof}
\subsection{Maximal local solutions to the Jacobi SDE}
In this section we draw a connection between  $\mathbf{M}$ and maximal local solutions to the Jacobi stochastic differential equation \eqref{eq3n}. As we will see, this is quite immediate whenever $\alpha,\beta >-1$, but requires some technical  work in any other case. We denote the minimum completed admissible filtration of $\mathbf{M}$ by $\mathfrak{F}$, its last element by $\mathfrak{F}_\infty$ and the completion of $\mathfrak{F}$ with respect to $P_x$ by $\mathfrak{F}^{P_x}$, for details see \cite[p.386]{fukushima2011}. Then for every $f\in D(L)$ the process  
\begin{equation}\label{eq13}
	\tilde{f}(Y_t) - \tilde{f}(Y_0)-\int_0^t L f(Y_s)\,ds\end{equation}
is a martingale under $P_x$ for quasi-every $x\in X$, see \cite[Remark 3.2, p.517]{Albeverio_Roeckner95}. We point out that  the convention $Lf(\Delta)=0$ is used again. For the remainder of this section we fix an $x\in X\cap \tilde{X}$ and choose $\Lambda\in \mathfrak{F}_\infty$ as a set of full measure under $P_x$ such that
\begin{enumerate}[label=(\roman*)]
	\item $Y_\cdot (\omega)\colon [0,\zeta(\omega))\to X$ is continuous,
	\item $Y_{\zeta-}(\omega)=\Delta$ if $\zeta(\omega)<\infty$ and
	\item $\zeta(\omega)>0$
\end{enumerate}
for every $\omega \in \Lambda$. Such a set exists due to Theorem \ref{THM_proc1} (i), (iv), (v) and the fact that $\mathbf{M}$ is a normal Markov process. We recall the functions $\mu$, $\nu$ introduced in Section \ref{Sec_local_sol} and define $u(t)$ and $v(t)$ as the solutions to the differential equation $u'(t)= \mu (u(t))$ with initial value $u(0)=d$ and $v(0)=0$, respectively. 
\begin{rem}\label{rem_containement}
	Under the assumption $\alpha \le -1$ we have $\mu(d)=a-bd\ge 0$ and therefore $u(t)\ge d$ for all $t\ge 0$. Similarly, $v(t)\le 0$  whenever $\beta \le -1$.
\end{rem}
We introduce the modified process
\begin{equation}\label{eq24}
Z_t(\omega)=\begin{cases}
	Y_t(\omega),&\omega \in \Lambda, t< \zeta(\omega),\\ 
	u(t-\zeta(\omega)),&\omega \in \Lambda, t\ge \zeta(\omega), \lim_{s\nearrow \zeta}Y_s(\omega)=d,\\
	v(t-\zeta(\omega)),&\omega \in \Lambda, t\ge \zeta(\omega), \lim_{s\nearrow \zeta}Y_s(\omega)=0,\\
	0, &\Omega \setminus \Lambda 
\end{cases}
\end{equation}
and the modified life time 
\begin{equation}\label{eq30}
\tilde{\zeta}(\omega)=\begin{cases} 
	\infty,&\omega \in \Lambda,\alpha =-1,\zeta<\infty,  \lim_{t\nearrow \zeta}Y_t(\omega)=d,\\
	\infty,&\omega \in \Lambda,\beta =-1,\zeta<\infty, \lim_{t\nearrow \zeta}Y_t(\omega)=0,\\
	0, &\omega\in \Omega \setminus \Lambda,
	\\\zeta(\omega),&\text{else}. 
\end{cases}
\end{equation}
Notice that we interpret here the limit $\lim_{t\nearrow \zeta}Y_t$ as an element of $\RR$ instead of the topological space $X_\Delta$, such that we can distinguish between the boundary points at which $Y$ dies. We make some technical observations.
\begin{lemma}\label{aux_statements}The following holds.
	\begin{enumerate}[label=(\roman*)]
		\item 	The process $Z$ has continuous paths, is $\mathfrak{F}^{P_x}$-adapted and satisfies
		$
			Z_{t\wedge \tilde{\zeta}}\in [0,d]$ for all $t\ge 0$.
		\item 
		$\tilde{\zeta}$ is a stopping time with respect to $\mathfrak{F}^{P_x}$.
		\item It holds $P_x$-almost surely $\tilde{\zeta}= \inf\{t\ge 0|Z_t\notin[0,d] \}$.
	\end{enumerate}
\end{lemma}
\begin{proof} The continuity assertion and the claim that $Z_{t\wedge \tilde{\zeta}}\in [0,d]$ of part (i) follow by the definition of $\Lambda$, $Z$ and $\tilde{\zeta}$ together with the observation, that $u= d$ ($v= 0$) is constant whenever $\alpha =-1$ ($\beta =-1$). The claim regarding adaptedness in (i) reduces to verifying that
	\begin{equation}\label{eq12n}
		\{u(t-\zeta)\in B\}\cap \{0<\zeta\le t\}\cap \{\lim_{s\nearrow \zeta}Y_s=d\}\;\in\; \mathfrak{F}_t^{P_x}
	\end{equation}
and 
\[
\{v(t-\zeta)\in B\}\cap \{0<\zeta\le t\}\cap \{\lim_{s\nearrow \zeta}Y_s=0\}\;\in\; \mathfrak{F}_t^{P_x}
\]
	for every $B\in \mathfrak{B}(\RR)$. We observe that
	\begin{equation}\label{eq15}
		\{
		u(t-t\wedge \zeta)\in B
		\}\;\in\;  \mathfrak{F}_t^{P_x},
	\end{equation}
	since $t\wedge \zeta$ is $\mathfrak{F}_t^{P_x}$-measurable. Moreover, we have 
	\begin{equation}\label{eq16}
		\left\{0<\zeta\le t\right\}\cap\left(\bigcup_{k\in \mathbb{N}} \bigcap_{n\ge k} \left\{Y_{t\wedge (\zeta\vee \frac{1}{n}-\frac{1}{n})}>\frac{d}{2}\right\}\right) \,\in \, \mathfrak{F}_t^{p_x} ,
	\end{equation}
	as a consequence of $Y_{t\wedge (\zeta\vee \frac{1}{n}-\frac{1}{n})}$ being $\mathfrak{F}_{t+\frac{1}{n}}^{p_x}$-measurable and the right-continuity of $\mathfrak{F}^{p_x} $, for the latter see \cite[Lemma A.2.2, p.386]{fukushima2011}. Since the left-hand side of \eqref{eq12n} is the intersection of \eqref{eq15} and \eqref{eq16}, \eqref{eq12n} follows. Similarly, (ii) reduces to showing that
	\begin{equation}\label{eq21n}
		\{0<\zeta\le t\}\cap\left\{
		\lim_{s\nearrow \zeta}Y_s=d
		\right\}\;\in\; \mathfrak{F}_t^{P_x}
	\end{equation} and 
	\[
		\{0<\zeta\le t\}\cap\left\{
	\lim_{s\nearrow \zeta}Y_s=0
	\right\}\;\in\; \mathfrak{F}_t^{P_x}.
	\]
	Both statements can be verified by rewriting the events analogously to \eqref{eq16}. To also verify (iii) we notice that as a consequence of part (i) we have 
	\[
	P_x(\{\tilde{\zeta } =\infty\}\cap\{\inf\{t\ge 0|Z_t\notin[0,d]\}=\infty\} )\,=\,
	P_x(\{\tilde{\zeta } =\infty\} ).
	\] Hence, it is sufficient to verify that $\tilde{\zeta}= \inf\{t\ge 0|Z_t\notin[0,d] \}$ $P_x$-almost surely on the set $\{\tilde{\zeta}<\infty\}$. Therefore,  we only have to consider the cases  $\alpha<-1$ or $\beta <-1$ by Theorem \ref{THM_proc1} (iii) and the definition of $\tilde{\zeta}$. We assume that $\alpha \ge -1$ and $\beta<-1$ and let $\omega \in \{\tilde{\zeta}<\infty\}\cap \Lambda$. It follows that 
$
	\lim_{s\nearrow \zeta}Y_s(\omega )=0
	$. Indeed for $\alpha >-1$ this follows by the choice of $\Lambda$ and for $\alpha =-1$ by definition of  $\tilde{\zeta}$ and the assumption $\omega\in \{\tilde{\zeta}<\infty\}$.
	Since $v(t)<0$ for  $t>0$ we conclude that
	\begin{equation}\label{eq17n}
		\inf\{t\ge 0| Z_t(\omega )\notin [0,d]\}=\zeta(\omega)=\tilde{\zeta}(\omega).
	\end{equation}It follows that
	\begin{equation}\label{eq18n}
		P_x(\{\tilde{\zeta } <\infty\}\cap\{\inf\{t\ge 0|Z_t\notin[0,d]\}=\tilde{\zeta }\} )=
		P_x(\{\tilde{\zeta } <\infty\} )
	\end{equation}
	as desired. The case $\alpha <-1$, $\beta \ge -1$ can be treated analogously. Lastly, we assume that $\alpha, \beta <-1$. In this case we additionally have that $u(t)>0$ for  $t>0$. Hence, for each $\omega \in \{\tilde{\zeta}<\infty\}\cap \Lambda$ we conclude again \eqref{eq17n}, which yields \eqref{eq18n}. This finishes the proof.
\end{proof}

In the final theorem of this section we prove that the tuple $(Z,\tilde{\zeta})$ is a maximal local solution to \eqref{eq4n} with initial value $x$. The main ingredients are Theorem \ref{extension_thm} and the martingale problem characterization of the auxiliary equation \eqref{eq4n}. To apply the latter we let $W'$ be a Brownian motion on a probability space $(\Omega', \mathfrak{A}', P')$ with respect to a right-continuous filtration $\mathfrak{F}'$. We define the enriched probability space by $\Omega^\dag= \Omega\times \Omega '$,  $P^\dag=P_\nu\times P'$ and the completed $\sigma$-field $\mathfrak{A}^\dag=\overline{ \mathfrak{F}^{P_x}_\infty\times \mathfrak{A}'}^{P^\dag}$. We equip it with the filtration $\mathfrak{F}^\dag$, which we define as the $P^\dag$-completion of $\mathfrak{F}_t^{P_x}\times \mathfrak{F}_t'$ in ${\mathfrak{A}}^\dag$. In particular, $\mathfrak{F}^\dag$ satisfies the usual conditions by \cite[Lemma 6.8, p.101]{kallenberg1997foundations}. We note that any random variable defined on $\Omega$ or $\Omega'$ extends canonically to  the enriched  space.
Moreover, we denote the generator of \eqref{eq4n} by $G$, i.e. we set
\begin{equation}\label{generatorG}
	Gf(x)= \frac{\nu^2(x)}{2}f''(x)+\mu(x)f'(x)
\end{equation}
for  $f\in C^2(\RR)$. Note that this is consistent with \eqref{eq12}, which was introduced for functions on $[0,d]$. 

\begin{thm}\label{is_sol_max}  There exists a Brownian motion $W^\dag$ on  $(\Omega^\dag,\mathfrak{A}^\dag,P^\dag)$ such that $Z$ is a solution to \eqref{eq4n} with inital value $x$. In particular $(Z,\tilde{\zeta})$ is a maximal local solution to \eqref{eq3n}.
\end{thm}
\begin{proof}
	If we prove the first part of the statement, the second one follows by Lemma \ref{aux_statements} (iii) together with Theorem \ref{extension_thm} (ii). For the first part it is again sufficient to show that $Z$ solves the $G$-martingale problem, i.e. that 
	\begin{equation}\label{eq55}	f(Z_t)-f(Z_0)-\int_0^t Gf(Z_s) ds\end{equation}
	is for every $f\in C_c^\infty(\RR)$ a martingale with respect to $\mathfrak{F}^{P_x}$, due to \cite[Theorem 3.3, p.293]{ethier2005markov}. To verify this we distinguish different cases. We assume first that $\alpha, \beta >-1$. In this case $\zeta=\infty$ $P_x$-almost surely by Theorem \ref{THM_proc1} (iii) and therefore $Y$ and $Z$ are indistinguishable. Since \eqref{eq13} is an $\mathfrak{F}^{P_x}$-martingale it follows that indeed $Z$ solves the $G$-martingale problem by Lemma \ref{gen_Lemma} (i). Secondly, we assume that	$\alpha >-1$ and $\beta \le -1$ such that $(L,D(L))\supset (G,C_c^\infty((0,d]))$ by Lemma \ref{gen_Lemma} (ii). Let $f\in C_c^\infty(\RR)$, then we decompose \eqref{eq55} into	\begin{equation}\label{eq91}
			f(Z_{{\zeta }\wedge t})-f(Z_0)-\int_0^{\zeta \wedge t} Gf(Z_s)\,ds\,+\,
			f(Z_t)-f(Z_{\zeta \wedge t})-\int^t_{\zeta \wedge t} Gf(Z_s)\,ds.
	\end{equation}
	By definition of $Z$ we can replace the last three terms of \eqref{eq91} by
		\[\mathbbm{1}_{\{t> \zeta\}}
	\left[f(v(t-\zeta))-f(v(0))-\int_0^{t-\zeta} \mu(v(s)) f'(v(s))\,ds\right],
	\]
	because $\nu=0$ on $(-\infty,0]$. Since $v'(s)= \mu(v(s))$ an  application of the fundamental theorem of calculus shows that the above term vanishes. To also treat the remaining part of \eqref{eq91} we introduce 
		\[
	\zeta_n=\inf\left\{t\ge 0\left| Y_t\le \frac{1}{n}\right.\right\},
	\]
	which is a stopping time for every $n$ by \cite[Theorem A.2.3, p.391]{fukushima2011}.
	It follows then that  $\lim_{n\to \infty} \zeta_n =\zeta$ $P_x$-almost surely by Theorem \ref{THM_proc1} (iv) and (v). Using additionally Lemma \ref{aux_statements} (i) we conclude that
		\begin{equation}\label{eq14n}M^{(n)}_t \,=\, f(Y_{{\zeta_n }\wedge t})-f(Y_0)-\int_0^{\zeta_n \wedge t} Gf(Y_s)\,ds\,=\, f(Z_{{\zeta_n }\wedge t})-f(Z_0)-\int_0^{\zeta_n \wedge t} Gf(Z_s)\,ds
	\end{equation}
	converges $P_x$-almost surely to the first three terms of \eqref{eq91} as $n\to \infty$. We can replace $f$ with a function $g\in C_c^\infty((0,d])$, which coincides with $f$ on $[-\frac{1}{n}, d]$, without changing the value of \eqref{eq14n}. It follows  by Lemma \ref{gen_Lemma} (ii) that $M^{(n)}$ is a stopped version of \eqref{eq13} and therefore an $\mathfrak{F}^{P_x}$-martingale for every $n\in \mathbb{N}$. Since moreover $M^{n}_t$ is uniformly bounded for fixed $t$ it follows that its limit is an $\mathfrak{F}^{P_x}$-martingale as well. We conclude that $Z$ indeed solves the $G$-martingale problem. The cases 	$\alpha \le -1$ and $\beta >-1$ as well as $\alpha, \beta \le -1$ can be treated analogously, the latter by using the approximating sequence of stopping times
	\[
	\zeta_n=\inf\left\{t\ge 0\left| Y_t\le \frac{1}{n} \vee Y_t\ge d-\frac{1}{n}\right.\right\}
	\]
	instead.
\end{proof}
\begin{rem}By definition, the process $Z$ and its life time $\tilde{\zeta}$ can be constructed from $Y$ and $\zeta$. Considering also uniqueness in law of maximal local solutions as shown in Corollary \ref{uniq_laws} one obtains properties of a general maximal local solution to \eqref{eq3n} by transferring the properties of $\mathbf{M}$.
\end{rem}
\subsection{Minimal local solutions to the Jacobi SDE}
In this last section we consider the restriction of the Hunt process $\mathbf{M}$ to the open interval $\hat{X} = (0,d)$. The restricted process  is obtained by defining the stopping time 
\[
\dot{\tau}_{X\setminus\hat{X}} \,=\, \inf \{t\ge 0 | Y_t\in X\setminus \hat{X}\}
\]
and stopping the process at this time, i.e. by setting 
\[
\hat{Y}_t(\omega)=\begin{cases}
	Y_t(\omega), &t< \dot{\tau}_{X\setminus\hat{X}}(\omega),\\\Delta , 
	& t\ge \dot{\tau}_{X\setminus\hat{X}}(\omega).
\end{cases}
\]
Then $\hat{\mathbf{M}}= (\Omega, \mathfrak{A}, (P_x)_{x\in \hat{X}_\Delta}, (\hat{Y}_t)_{t\ge 0}) $ is a Hunt process again, see \cite[Theorem A.2.10, p.400]{fukushima2011}. We note that, as shown in the proof of the cited theorem, $\hat{\mathbf{M}}$ is quasi-left continuous and a strong Markov process with respect to the minimum completed admissible filtration $\mathfrak{F}$ of $\mathbf{M}$. Its life time is given by $\hat{\zeta}= \zeta \wedge \dot{\tau}_{X\setminus\hat{X}} $ and its transition function by 
\begin{equation}\label{eq31}
\hat{\rho}_t(x,B) \, = \, P_x(\{X_t \in B \wedge  t<  \dot{\tau}_{X\setminus\hat{X}}\} )
\end{equation}
for $B\in \mathfrak{B}(\hat{X})$.
We restrict the Dirichlet form $(\EE, D(\EE))$ as well by replacing its domain by 
\[
\{
f\in D(\EE) | \tilde{f} = 0 \text{ q.e. on } X\setminus \hat{X}
\}.
\]
We denote the restricted form by $(\hat{\EE}, D(\hat{\EE}))$ and obtain as a consequence of Theorem \ref{DE} and Theorem \ref{THM_proc1} (i) that $(\hat{\EE}, D(\hat{\EE}))= (\EE, \FF)$. In particular, it is the form corresponding to the Friedrichs extension of the operator $(G, C_c^\infty((0,d)))$.  By \cite[Theorem 4.4.3 (i), p.174]{fukushima2011} $(\hat{\EE}, D(\hat{\EE}))$ is a regular Dirichlet form on $L^2(\hat{X}, dm)$. Since $\hat{X}$ is an open subset of $X$ \cite[Theorem 4.4.2, pp.173-174]{fukushima2011} yields that $(\hat{\EE}, D(\hat{\EE}))$ is associated to $\hat{\mathbf{M}}$. As a consequence of Theorem \ref{DE} the form $(\hat{\EE}, D(\hat{\EE}))$ is only different from $(\EE, D(\EE))$ if $-1<\alpha < 0 $ or $-1<\beta <0$. The same holds for the corresponding processes.
\begin{prop}We assume that neither $-1<\alpha < 0 $ nor $-1<\beta <0$.
Then the processes $Y$ and $\hat{Y}$ are indistinguishable under $P_x$ 	for every $x\in \hat{X}$. In particular, the properties stated in Theorem \ref{THM_proc1} hold also for $\hat{\mathbf{M}}$.
\end{prop}
\begin{proof}
	It suffices to verify that $P_x(\{\dot{\tau}_{X\setminus \hat{ X}}<\infty\})=0$ for every $x\in \hat{X}$. If $\alpha, \beta \le -1$ we have $X=\hat{X}$ such that the above is trivial. In any of the remaining cases the set $X\setminus \hat{X}$ is properly exceptional with respect to $\mathbf{M}$ by Theorem \ref{THM_proc1} (i) and \cite[Theorem 4.1.1, p.155]{fukushima2011} and therefore the desired equality follows.
\end{proof}
In any other case, we  conclude the following.
\begin{thm} Let $-1<\alpha <0$ or $-1<\beta <0$, then the  following holds.
	\begin{enumerate}[label=(\roman*)]
		\item
		The only subset of $\hat{X}$, which is exceptional with respect to $\hat{\mathbf{M}}$, is the empty set.
		\item The transition probability $\hat{\rho}_t(x,\cdot)$ is absolutely continuous with respect to $dm$ for every $t>0$ and $x\in \hat{X}$.
		\item The path $[0, \hat{\zeta})\to \hat{ X}, t\mapsto \hat{Y}_t$ is $P_x$-almost surely continuous for  every $x\in \hat{X}$.
		\item It holds $P_x\left(\left\{\hat{Y}_{\hat{\zeta}-} \in \hat{X} \wedge \hat{\zeta}  <\infty\right\}\right)=0$ for  every $x\in\hat{ X}$.
		\item We have 
		for every $x\in \hat{X}$ that \begin{equation*}
		P_x(\{\hat{\zeta}<\infty\})\, \ge \, 
	\begin{cases}
		1,& \alpha, \beta>-1,\\
		\frac{\Gamma\left(-\alpha-\beta\right)}{\Gamma(1-\beta)\Gamma(-\alpha)}\,\left(\frac{x}{d}\right)^{-\beta}
		{}_2F_1\left(
		\alpha+1,-\beta;1-\beta;\frac{x}{d}
		\right),& \alpha>-1,\beta\le -1,\\
		\frac{\Gamma\left(-\alpha-\beta\right)}{\Gamma(1-\alpha)\Gamma(-\beta)}\,\left(1-\frac{x}{d}\right)^{-\alpha}
		{}_2F_1\left(
		\beta+1,-\alpha;1-\alpha;1-\frac{x}{d}
		\right),& \beta>-1, \alpha\le -1.
	\end{cases}\end{equation*}
	\end{enumerate}
\end{thm}
\begin{proof}Part (i) follows from Theorem \ref{THM_proc1} (i) and \cite[Theorem 4.4.3 (ii)]{fukushima2011}. Part (ii) follows from \eqref{eq31} together with Theorem \ref{THM_proc1} (ii). Part (iii) is a direct consequence of Theorem \ref{THM_proc1} (iii) and the definitions of $\hat{Y}$ and $\hat{\zeta}$. We decompose the probability from (iv) into the parts
	\[
	P_x\left(\left\{\hat{Y}_{\hat{\zeta}-} \in \hat{X} \wedge \hat{\zeta} = \zeta  <\infty\right\}\right) \,+\,
	P_x\left(\left\{\hat{Y}_{\hat{\zeta}-} \in \hat{X} \wedge \hat{\zeta}  <\zeta\right\}\right).
	\]
	The first term vanishes by Theorem \ref{THM_proc1} (v). The second one vanishes as a consequence of part (iv) of the same Theorem. Part (v) follows from Theorem \ref{Prob_hitting} and the fact that $
	\hat{\zeta}\le \tau_{\{d\}}\wedge \tau_{\{0\}}$.
\end{proof}

Finally, we prove a statement similar to Theorem \ref{is_sol_max} relating the restricted process $\hat{\mathbf{M}}$ to  minimal solutions to the Jacobi stochastic differential equations. Therefore let $x\in \hat{X}$ and  $\Lambda\in \mathfrak{F}_\infty$ and $Z$ be the corresponding set and process from the last section.
\begin{lemma}\label{aux_statementIIn}
	It holds the following.
	\begin{enumerate}[label=(\roman*)]\item We have $P_x$-almost surely for every $t\ge 0$ that
		\begin{equation}\label{eq20n}
			Z_{t\wedge {\hat{\zeta}}}=\begin{cases}
				\hat{Y}_t, & t< \hat{\zeta},\\
				d, & t\ge  \hat{\zeta}, \lim_{t\nearrow \hat{\zeta} } \hat{Y}_t=d,\\
				0, & t\ge  \hat{\zeta}, \lim_{t\nearrow \hat{\zeta}} \hat{Y}_t=0.\\
			\end{cases}
		\end{equation}
		\item We have that $\hat{\zeta}= \inf\{t\ge 0|Z_t\in \{0,d \}\}$  $P_x$-almost surely.
	\end{enumerate}
\end{lemma}
\begin{proof} Due to $u(0),v(0)\in \{0,d\}$, part (ii) follows by the definition of $\hat{\zeta}$ and  $Z$. Since we have
	\[
	\mathbbm{1}_{\{t<\hat{\zeta}\}}Z_t\,=\,
	\mathbbm{1}_{\{t<\hat{\zeta}\}}Y_t\,=\,
	\mathbbm{1}_{\{t<\hat{\zeta}\}}\hat{Y}_t
	\]
	$P_x$-almost surely and $Z$ has continuous paths, the right-hand side of \eqref{eq20n} is nothing but the process $Z$ stopped as it hits the set $\{0,d\}$. Therefore, (i) is a consequence of (ii). 
\end{proof}
Finally, we consider  the enriched probability space $(\Omega^\dag,\mathfrak{A}^\dag,P^\dag)$ with the filtration  $\mathfrak{F}^\dag$.

\begin{cor}$(Z, \hat{\zeta})$ is a minimal local solution to \eqref{eq3n} with initial value  $x$.
\end{cor}
\begin{proof}By Theorem \ref{is_sol_max} the process $Z$ is a solution to \eqref{eq4n} with initial value $x$ on $(\Omega^\dag,\mathfrak{A}^\dag,P^\dag)$ with respect to a Brownian motion $W^\dag$. The claim follows by Lemma \ref{aux_statementIIn} (ii) and Theorem \ref{extension_thm} (i). 
\end{proof}
\begin{rem}
	By  \eqref{eq20n} and the uniqueness in law statement from Corollary \eqref{uniq_laws} the properties of a general minimal local solution to \eqref{eq3n} can be derived from the properties of $\hat{\mathbf{M}}$.
\end{rem}
	\appendix
	\section{Appendix}
	
	\subsection{A localized Yamada-Watanabe condition}
	We provide a localized version of the Yamada Watanabe condition. The proof translates verbatim from the classical setting, see \cite[Theorem 20.3, p.374]{kallenberg1997foundations}, and is contained for convinience of the reader.
	\begin{lemma}\label{yamada_watanabe}
		Let $\mu, \nu \colon \mathbb{R}\to \mathbb{R}$ be mappings such that $\mu$ is Lipschitz and $\nu$ is $\frac{1}{2}$-H\"older continuous. Moreover, let $(\Omega, \mathfrak{A},P)$ be a probability space equipped with a filtration satisfying the usual conditions and $W$ a Brownian motion. If there are two adapted processes  $Y^{(1)}$ and $Y^{(2)}$ and a stopping time $\zeta$ such that
		\begin{enumerate}[label=(\roman*)]
			\item $Y^{(1)}_0=Y^{(2)}_0$,
			\item $Y^{(i)}_{\cdot \wedge \zeta}$ has continuous paths and 
			\item $Y^{(i)}_{t \wedge \zeta} = \int_0^{t \wedge \zeta}\mu(Y^{(i)})ds + \int_0^{t \wedge \zeta}\nu(Y^{(i)})dW_s$ for all $t\ge 0$, $i\in \{1,2\}$,
		\end{enumerate}
		then $Y^{(1)}_{t \wedge \zeta}=Y^{(2)}_{t \wedge \zeta}$ for all $t\ge 0$.
	\end{lemma}
	\begin{proof}We define the process 
		\begin{equation*}
			D_t \,=\,	Y_{t\wedge\zeta}^{(1)}\,-\,Y_{t\wedge\zeta}^{(2)} \,=\,  \int_0^{t\wedge\zeta} \mu(Y^{(1)}_s)-\mu(Y^{(2)}_s)\,ds 
			\,+\, \int_0^{t\wedge\zeta}  \nu(Y^{(1)}_s)-\nu(Y^{(2)}_s)\,dW_s.
		\end{equation*}
		It suffices to show that $D=0$. By a localization argument we can assume that $D$ is uniformly bounded. Let $(L^x_t)_{t\ge 0, x\in \RR}$ be a cadlag version of the local time of $D$, for existence of such a version  see \cite[Theorem 19.4, p.352]{kallenberg1997foundations}. Then \cite[Theorem 19.5, p.353]{kallenberg1997foundations} yields for any $t\ge 0$ that 
		\[
		\int_{-\infty}^\infty f(x) L_t^x\, dx \,=\, 
		\int_0^t f(D_{s})\, d\left<D\right>_s
		\,=\, 
		\int_0^{t\wedge \zeta} f(Y_s^{(1)}-Y_s^{(2)}) (\nu(Y_s^{(1)})-\nu(Y_s^{(2)}))^2\, ds\,\le \, C^2t,
		\]
		where we choose
		$f(x)=\frac{1}{|x|}$ for $x\ne 0$ and $f(0)=1$ and $C$ to be the $\frac{1}{2}$-H\"older seminorm of $\nu$.
		By taking the expectation we obtain that
		\[
		\infty\,>\, E\left[
		\int_{-\infty}^\infty f(x) L_t^x\, dx
		\right] \, = \,  \int_{-\infty}^\infty f(x) E\left[
		L_t^x
		\right]\, dx.
		\]
		Due to the right-continuity of $L_t^x$ in $x$ we can employ Fatou's lemma to conclude  that
		\[
		E(L^0_t)\, \le \, \liminf_{x\searrow 0} E(L^x_t) \,=\,0.
		\]
		Using the defining property of the local time we obtain
		\[
		|D_t|\, =\, \int_0^{t\wedge \zeta}\sgn(D_s)(\mu(Y^{(1)}_s)-\mu(Y^{(2)}_s) )\,ds\,+\, \int_0^{t\wedge \zeta}\sgn(D_s)  (\nu(Y^{(1)}_s)-\nu(Y^{(2)}_s))\,dW_s,
		\]
		where the convention $\sgn(0)=-1$ is used.
		We note that the integrand of the stochastic integral is uniformly bounded by the boundedness of $D$ and the H\"older continuity of $\nu$. Therefore, the stochastic integral is a martingale and  taking the expectation yields that
		\[
		E(|D_t|)\, \le\, C\int_0^{t}E(|D_s|)\,ds,
		\]
		where we let here $C$ be  the Lipschitz coefficient of $\mu$.
		Since this holds for any $t\ge 0$ it is left to apply Gr\"onwall's Lemma.
	\end{proof}

	\subsection{Boundary values of hypergeometric functions}\label{app_BV}

	In this section we perform the tedious steps leading to Lemma \ref{prop_f1}.
	\begin{proof}[Proof of Lemma \ref{prop_f1} (i)]
		This is a direct consequence of the definition of $\xi_\lambda$.
	\end{proof}

	\begin{proof}[Proof of Lemma \ref{prop_f1} (ii)]
		First, we consider the case $\beta >-1$. To  make use of Lemma \ref{gauss_limit} we distinguish between different signs of $-\alpha$.
		If $\alpha<0$ the identity \eqref{eq17} applies and yields
		\begin{equation}\label{eq179}
			\lim_{x\nearrow d} \,\xi_\lambda(x)\,=\,\frac{\Gamma(\beta+1)\Gamma(-\alpha)}{\Gamma\left(\frac{-\alpha+\beta+1}{2}+\gamma\right)\Gamma\left(\frac{-\alpha+\beta+1}{2}-\gamma\right)}.
		\end{equation}
		To conclude that this limit is positive we first note that its numerator is positive by the assumptions on $\alpha, \beta$ and the fact that $\Gamma(z)>0$ for $z>0$. Furthermore, as a consequence of \eqref{eq21} we have that
		\[
		\left(
		\frac{-\alpha+\beta+1}{2}\right)^2-\gamma^2\, >\, \left(\frac{\alpha+\beta+1}{2}\right)^2-\gamma^2\,>\,0.
		\]
		It follows that either $\gamma$ is  purely imaginary or satisfies $0\le \gamma< \frac{-\alpha+\beta+1}{2}$.
		In the latter case positivity of the denominator follows as for the numerator.
		In the former case positivity follows instead by $\Gamma(\bar{z})= \overline{\Gamma(z)}$ and that $\Gamma$ is non-zero outside of the negative real axis.
		Next, we assume that $\alpha=0$. Then \eqref{eq18} applies and yields 
		\[
		\lim_{x\nearrow d}\frac{\xi_\lambda(x)}{-\log\left(1-\frac{x}{d}\right)}\,=\,
		\frac{\Gamma(\beta+1)}{\Gamma\left(
			\frac{\beta+1}{2}+\gamma\right)\Gamma\left(
			\frac{\beta+1}{2}-\gamma\right)}.
		\]
		By analogous arguments  as before we  conclude that the right-hand side is positive. Since $-\log\left(1-\frac{x}{d}\right)$ converges to infinity as $x\nearrow d$, it follows that  $\lim_{x\nearrow d}\xi_\lambda(x)= \infty$.
		Lastly, we assume $\alpha >0$. The identity \eqref{eq26}  gives us then
		\begin{equation*}
			\lim_{x\nearrow d}\frac{\xi_\lambda(x)}{\left(1-\frac{x}{d}\right)^{-\alpha}}\,=\,
			\frac{\Gamma(\beta+1)\Gamma(\alpha)}{\Gamma\left(
				\frac{\alpha+\beta+1}{2}+\gamma\right)\Gamma\left(
				\frac{\alpha +\beta+1}{2}-\gamma\right)}.
		\end{equation*}
		The same reasoning as before applies to argue that the right-hand side is positive. Since $\left(1-\frac{x}{d}\right)^{-\alpha}$ approaches infinity as $x\nearrow d$, we obtain  $\lim_{x\nearrow d}\xi_\lambda(x)= \infty$.
		
		Now we consider the case $\beta \le -1$, where it is sufficient to consider the hypergeometric part of $\xi_\lambda$. To make   use of Theorem \ref{gauss_limit} we distinguish between the signs of 
		$-\alpha$ again. If $\alpha <0$,
		\eqref{eq17} gives us
		\begin{equation*}\label{eq180}
			\lim_{x\nearrow d} \,\xi_\lambda(x)\,=\,\frac{\Gamma(1-\beta)\Gamma(-\alpha)}{\Gamma\left(\frac{-\alpha-\beta+1}{2}+\gamma\right)\Gamma\left(\frac{-\alpha-\beta+1}{2}-\gamma\right)}.
		\end{equation*}
		If $\alpha=0$, \eqref{eq18} applies and yields 
		\[
		\lim_{x\nearrow d}\frac{\xi_\lambda(x)}{-\log\left(1-\frac{x}{d}\right)}\,=\,
		\frac{\Gamma(1-\beta)}{\Gamma\left(
			\frac{1-\beta}{2}+\gamma\right)\Gamma\left(
			\frac{1-\beta}{2}-\gamma\right)}.
		\]
		Finally, if $\alpha> 0$, we get by \eqref{eq26} that
		\begin{equation*}
			\lim_{x\nearrow d}\frac{\xi_\lambda(x)}{\left(1-\frac{x}{d}\right)^{-\alpha}}\,=\,
			\frac{\Gamma(1-\beta)\Gamma(\alpha)}{\Gamma\left(
				\frac{\alpha-\beta+1}{2}+\gamma\right)\Gamma\left(
				\frac{\alpha -\beta+1}{2}-\gamma\right)}.
		\end{equation*}
		Analogous considerations as in the case $\beta>-1$ yield positivity of these three limits. In particular, for $\alpha \ge 0$ we can conclude  that $\lim_{x\nearrow d}\xi_\lambda(x)= \infty$.
	\end{proof}
	Before proceeding with the remaining statements we note that the product rule, termwise differentiation of the hypergeometric function  as well as the identities \eqref{eq21} and
	\begin{equation*}
		\left(\frac{\alpha-\beta+1}{2}\right)^2-\gamma^2=\left(\frac{\alpha+\beta+1}{2}-\beta\right)^2-\gamma^2\,=\,
		\frac{2\lambda }{\sigma^2}-\beta(\alpha+1)
	\end{equation*}
yield the explicit expression
 	\begin{equation*}\label{eq23}\xi_\lambda'(x)\,=\,\begin{cases*}
 		\frac{2\lambda}{\sigma^2d(\beta+1)}\,{}_2F_1
 		\left(
 		\frac{\alpha+\beta+3}{2}+\gamma,
 		\frac{\alpha+\beta+3}{2}-\gamma;\beta+2;\frac{x}{d}
 		\right),&$\beta >-1$,\\
 		\frac{\frac{2\lambda }{\sigma^2}-\beta(\alpha+1)}{d(1-\beta)}\left(\frac{x}{d}\right)^{-\beta}
 		{}_2F_1\left(
 		\frac{\alpha-\beta+3}{2}+\gamma,
 		\frac{\alpha-\beta+3}{2}-\gamma;2-\beta;\frac{x}{d}
 		\right)\\-\frac{\beta}{d} \left(\frac{x}{d}\right)^{-(\beta+1)}{}_2F_1\left(
 		\frac{\alpha-\beta+1}{2}+\gamma,
 		\frac{\alpha-\beta+1}{2}-\gamma;1-\beta;\frac{x}{d}\right),&$\beta\le -1$\end{cases*}\end{equation*}	
 for $x\in(0,d)$.

	\begin{proof}[Proof of Lemma \ref{prop_f1} (iii)]
		For $\beta >-1$ the claim follows, because $cm(x)$ converges to  $0$ as $x\searrow 0$.
		If $\beta\le -1$, the derivative of $\xi_\lambda$ consists out of two summands.
		For the first one we observe that
		\[\frac{\sigma^2x^{\beta+1}(d-x)^{\alpha+1}}{2d^{\alpha+\beta+1}}\,	\left(\frac{x}{d}\right)^{-\beta}
		\,{}_2F_1\left(
		\frac{\alpha-\beta+3}{2}+\gamma,
		\frac{\alpha-\beta+3}{2}-\gamma;2-\beta;\frac{x}{d}
		\right)\,\to \,0
		\]
		as $x\searrow 0$. Hence,
		\begin{align*}&\lim_{x\searrow 0}\,\xi_\lambda'cm(x)
			\\=\,&\lim_{x\searrow 0}\,-\frac{\beta \sigma^2(d-x)^{\alpha+1}}{2d^{\alpha+1}} \,{}_2F_1\left(
			\frac{\alpha-\beta+1}{2}+\gamma,
			\frac{\alpha-\beta+1}{2}-\gamma;1-\beta;\frac{x}{d}\right)\,=\,
			\frac{-\beta \sigma^2}{2}.
		\end{align*}
	\end{proof}
	\begin{proof}[Proof of Lemma \ref{prop_f1} (iv)]
		We consider   $\beta >-1$. In this case we have
		\begin{align*}&
			\lim_{x\nearrow d} \,\xi_\lambda'cm(x)\\=\;&
			\lim_{x\nearrow d} \,\frac{\lambda x^{\beta+1}(d-x)^{\alpha+1}}{d^{\alpha+\beta+2}(\beta+1)}\,{}_2F_1
			\left(
			\frac{\alpha+\beta+3}{2}+\gamma,
			\frac{\alpha+\beta+3}{2}-\gamma;\beta+2;\frac{x}{d}
			\right)\\=\;&
			\frac{\lambda }{\beta+1}\,\lim_{x\nearrow d}\,\left(
			1-\frac{x}{d}
			\right)^{\alpha+1}
			{}_2F_1
			\left(
			\frac{\alpha+\beta+3}{2}+\gamma,
			\frac{\alpha+\beta+3}{2}-\gamma;\beta+2;\frac{x}{d}
			\right).
		\end{align*}
		To make use of Theorem \ref{gauss_limit} we distinguish between  the signs of
		$-(\alpha+1)$.
		We assume first $\alpha>-1$, i.e. $-(\alpha+1)< 0$. Then \eqref{eq26} applies and yields
		\[
		\lim_{x\nearrow d} \,\xi_\lambda'cm(x)\,=\,\frac{\lambda \Gamma(\beta+2)\Gamma(\alpha+1)}{(\beta+1)\Gamma\left(
			\frac{\alpha+\beta+3}{2}+\gamma
			\right)
			\Gamma\left(
			\frac{\alpha+\beta+3}{2}-\gamma
			\right)}.
		\]
		The claimed identity follows since $\Gamma(\beta+2)=(\beta+1)\Gamma(\beta+1)$. The numerator of the limit is positive, again by to our assumptions on $\alpha, \beta$. Note that
		\[
		\left(
		\frac{\alpha+\beta+3}{2}\right)^2-\gamma^2\, >\, \left(\frac{\alpha+\beta+1}{2}\right)^2-\gamma^2\,>\,0
		\]
		due to \eqref{eq21}. Hence either $\gamma$ is purely imaginary or $0\le \gamma < \frac{\alpha+\beta+3}{2}$ and positivity of the denominator follows as in the proof of Lemma \ref{prop_f1} \ref{r_boundary}.
		Next we assume that $\alpha=-1$, i.e. $-(\alpha+1)=0$.
		Then \eqref{eq18}
		gives us that
		\[
		\lim_{x\nearrow d}\frac{{}_2F_1
			\left(
			\frac{\beta+2}{2}+\gamma,
			\frac{\beta+2}{2}-\gamma;\beta+2;\frac{x}{d}
			\right)}{-\log\left(1-\frac{x}{d}\right)}\,=\,
		\frac{ \Gamma(\beta+2)}{\Gamma\left(
			\frac{\beta+2}{2}+\gamma
			\right)
			\Gamma\left(
			\frac{\beta+2}{2}-\gamma
			\right)}.
		\]
		Positivity of the right-hand side follows as before and consequently we have $\lim_{x\nearrow d} \xi_\lambda'cm(x)=\infty$.
		Lastly, we assume that $\alpha<-1$, i.e. $-(\alpha+1)>0$.
		Applying \eqref{eq17} results in
		\[
		\lim_{x\nearrow d} \,{}_2F_1
		\left(
		\frac{\alpha+\beta+3}{2}+\gamma,
		\frac{\alpha+\beta+3}{2}-\gamma;\beta+2;\frac{x}{d}
		\right)\,=\,\frac{\Gamma\left(\beta+2\right)\Gamma\left(-(\alpha+1)\right)}{\Gamma\left(\frac{-\alpha+\beta+1}{2}+\gamma\right)
			\Gamma\left(\frac{-\alpha+\beta+1}{2}-\gamma\right)}
		\] and
		analogous arguments as before yield that the right-hand side is positive. Since $\left(
		1-\frac{x}{d}
		\right)^{\alpha+1}$ approaches infinity as $x\nearrow d$ we conclude  $\lim_{x\nearrow d} \xi_\lambda'cm(x)=\infty$ also in this case.
		
		Now we consider $\beta\le -1$. Then we have
		\begin{align}\begin{split}\label{eq35}	
				&\lim_{x\nearrow d}\,
				\xi_\lambda'cm(x)\\=\;&\lim_{x\nearrow d}\,\left[\frac{2\lambda}{\sigma^2}-\beta(\alpha+1)\right]\frac{\sigma^2 (d-x)^{\alpha+1}}{2(1-\beta)d^{\alpha+1}}
				\,{}_2F_1\left(
				\frac{\alpha-\beta+3}{2}+\gamma,
				\frac{\alpha-\beta+3}{2}-\gamma;2-\beta;\frac{x}{d}
				\right)\\
				&-\,\frac{\sigma^2\beta(d-x)^{\alpha+1}}{2d^{\alpha+1}}\,{}_2F_1\left(
				\frac{\alpha-\beta+1}{2}+\gamma,
				\frac{\alpha-\beta+1}{2}-\gamma;1-\beta;\frac{x}{d}\right).
			\end{split}
		\end{align} We start by investigating the limit of
		\begin{gather}\label{eq199}-
			\,\frac{\sigma^2\beta}{2}\left(1-\frac{x}{d}\right)^{\alpha+1}\,{}_2F_1\left(
			\frac{\alpha-\beta+1}{2}+\gamma,
			\frac{\alpha-\beta+1}{2}-\gamma;1-\beta;\frac{x}{d}\right)
		\end{gather}
		as $x\nearrow d$. To 
		make use of Theorem \ref{gauss_limit} we distinguish between the signs of $-\alpha$.
		If $\alpha>0$, \eqref{eq26} yields that the limit
		\[\lim_{x\nearrow d}\,
		\left(1-\frac{x}{d}\right)^\alpha
		{}_2F_1\left(
		\frac{\alpha-\beta+1}{2}+\gamma,
		\frac{\alpha-\beta+1}{2}-\gamma;1-\beta;\frac{x}{d}\right)
		\]
		exists. Because of the additional factor $\left(1-\frac{x}{d}\right)$ the term \eqref{eq199} converges to $0$ as $x\nearrow d$. If $\alpha=0$, \eqref{eq18} yields that the limit
		\[
		\lim_{x\nearrow d}\,\frac{{}_2F_1\left(
			\frac{\alpha-\beta+1}{2}+\gamma,
			\frac{\alpha-\beta+1}{2}-\gamma;1-\beta;\frac{x}{d}\right)}{-\log\left(1-\frac{x}{d}\right)}
		\]
		exists and since $\lim_{x\nearrow d}\left(1-\frac{x}{d}\right)\log\left(1-\frac{x}{d}\right)=0$   we get that \eqref{eq199} converges to $0$ for $x\nearrow d$.
		If  $\alpha <0$, \eqref{eq17} implies that
		\begin{align}\begin{split}\label{eq195}&
				\lim_{x\nearrow d}\,{}_2F_1\left(
				\frac{\alpha-\beta+1}{2}+\gamma,
				\frac{\alpha-\beta+1}{2}-\gamma;1-\beta;\frac{x}{d}\right)\,\\=\;&\frac{\Gamma\left(1-\beta\right)\Gamma\left(-\alpha\right)}{\Gamma\left(\frac{-\alpha-\beta+1}{2}+\gamma\right)
					\Gamma\left(\frac{-\alpha-\beta+1}{2}-\gamma\right)}.
			\end{split}
		\end{align}
		For $-1<\alpha<0$ we  conclude as before that \eqref{eq199} converges to $0$ as $x\nearrow d$ because of the additional prefactor. For $\alpha=-1$ the term \eqref{eq199} instead converges to a real number which we do not specify  and instead denote by $y$ during this proof. For $\alpha<-1$  the prefactor  $\left(1-\frac{x}{d}\right)^{\alpha+1}$ converges to infinity as $x\nearrow d$. Since the right-hand side of \eqref{eq195} is positive by analogous arguments as before we see that \eqref{eq199} tends to to infinity as $x\nearrow d$ as well. We combine all cases in the following formula.
		\begin{equation}\label{eq183}\lim_{x\nearrow d}\frac{-\sigma^2\beta}{2}\left(1-\frac{x}{d}\right)^{\alpha+1}{}_2F_1\left(
			\frac{\alpha-\beta+1}{2}+\gamma,
			\frac{\alpha-\beta+1}{2}-\gamma;1-\beta;\frac{x}{d}\right)= 
			\begin{cases*}
				0,&$\alpha >-1$,\\
				y,&$\alpha=-1$,\\
				\infty,&$\alpha <-1$.
			\end{cases*}
		\end{equation}
		We proceed by investigating the limiting behavior of 
		\begin{equation}\label{eq200}
			\left[\frac{2\lambda}{\sigma^2}-\beta(\alpha+1)\right]\frac{\sigma^2 (d-x)^{\alpha+1}}{2(1-\beta)d^{\alpha+1}}
			\,{}_2F_1\left(
			\frac{\alpha-\beta+3}{2}+\gamma,
			\frac{\alpha-\beta+3}{2}-\gamma;2-\beta;\frac{x}{d}
			\right)\end{equation}
		as $x\nearrow d$ using Theorem \ref{gauss_limit} and therefore distinguish between the signs of $ -(\alpha+1)$.
		
		For $\alpha >-1$, i.e. $-(\alpha+1)<0$, we get by \eqref{eq26} that
		\begin{align*}
			&
			\lim_{x\nearrow d}\,\left[\frac{2\lambda}{\sigma^2}-\beta(\alpha+1)\right]\frac{\sigma^2 (d-x)^{\alpha+1}}{2(1-\beta)d^{\alpha+1}}
			{}_2F_1\left(
			\frac{\alpha-\beta+3}{2}+\gamma,
			\frac{\alpha-\beta+3}{2}-\gamma;2-\beta;\frac{x}{d}
			\right)\\
			=\;\;&\left[\frac{2\lambda}{\sigma^2}-\beta(\alpha+1)\right]\frac{\sigma^2\Gamma(2-\beta)\Gamma(\alpha+1)}{2(1-\beta)\Gamma\left(\frac{\alpha-\beta+3}{2}+\gamma\right)\Gamma\left(\frac{\alpha-\beta+3}{2}-\gamma\right)}.
		\end{align*}
		Substituting $\Gamma(2-\beta)=(1-\beta)\Gamma(1-\beta)$ together with \eqref{eq183} yields the claimed identity in this  case.
		Note that the fraction on the right-hand side is positive by analogous reasoning as before. Also the prefactor is positive in this particular case of $\alpha$ and $\beta$.
		
		Next, we consider the case $\alpha=-1$. Then \eqref{eq18} yields	\[
		\lim_{x\nearrow d}\,\frac{{}_2F_1\left(
			\frac{2-\beta}{2}+\gamma,
			\frac{2-\beta}{2}-\gamma;2-\beta;\frac{x}{d}
			\right)}{-\log\left(1-\frac{x}{d}\right)}\,=\,\frac{\Gamma\left(2-\beta\right)}{\Gamma\left(\frac{2-\beta}{2}+\gamma\right)\Gamma\left(\frac{2-\beta}{2}+\gamma\right)}.
		\]
		By analogous arguments as before we  conclude that this limit is positive.
		Since also the prefactor of the expression \eqref{eq200} is positive due to $\alpha=-1$ it follows that \eqref{eq200} tends to infinity as $x\nearrow d$. Together with \eqref{eq183} this implies that \eqref{eq35} equals infinity in this case.
		
		It remains to consider $\alpha <-1$, i.e. $-(\alpha+1)>0$. 
		Then \eqref{eq17}  yields
		\[
		\lim_{x\nearrow d}\,{}_2F_1\left(
		\frac{\alpha-\beta+3}{2}+\gamma,
		\frac{\alpha-\beta+3}{2}-\gamma;2-\beta;\frac{x}{d}
		\right)\,=\,\frac{\Gamma\left(2-\beta\right)\Gamma(-(\alpha+1))}{\Gamma\left(\frac{-\alpha-\beta+1}{2}+\gamma\right)\Gamma\left(\frac{-\alpha-\beta+1}{2}+\gamma\right)}.
		\]
		Analogously as before we get that the right-hand side of the above equality is positive. By $\lim_{x\nearrow d}(d-x)^{\alpha+1}= \infty$  it is sufficient 
		to show $\frac{2\lambda}{\sigma^2}-\beta(\alpha+1)>0$ to conclude that \eqref{eq200} tends to infinity as $x\nearrow d$. We assumed in this particular case that \eqref{ass_r2} holds which implies that
		\[
		\frac{2\lambda}{\sigma^2}-\beta(\alpha+1)\,>\,
		\left(
		\frac{\alpha+\beta+1}{2}
		\right)^2 -\beta(\alpha+1)\,=\,\left(
		\frac{\beta-(\alpha+1)}{2}
		\right)^2\,\ge \,0.
		\]
		The claimed identity follows by employing \eqref{eq183}.
	\end{proof}

	\section*{Acknowledgments}
	The second author thanks Mark Veraar for pointing out relevant literature on Hardy's inequality.
	
	\printbibliography
	
\end{document}